\documentclass[11pt,a4paper,oneside]{amsart}

\newcounter{commentcounter}

\usepackage{amsmath} 
\usepackage{amssymb} 
\usepackage{amsthm} 
\usepackage{stmaryrd} 
\usepackage[english]{babel} 
\usepackage[font=small,justification=centering]{caption} 
\usepackage[nodayofweek]{datetime}
\usepackage[shortlabels]{enumitem} 
\usepackage[T1]{fontenc} 
\usepackage[utf8]{inputenc} 
\usepackage{ifthen} 
\usepackage{mathabx} 
\usepackage{mathtools} 
\usepackage[dvipsnames]{xcolor} 
\usepackage[pdftex,  colorlinks=true]{hyperref} 
    \hypersetup{urlcolor=RoyalBlue, linkcolor=RoyalBlue,  citecolor=black}
\usepackage{setspace} 
\usepackage{tikz-cd} 
\usepackage{xfrac} 
\usepackage[capitalize]{cleveref} 

\makeatletter
\@namedef{subjclassname@1991}{Mathematical subject classification 1991}
\@namedef{subjclassname@2000}{Mathematical subject classification 2000}
\@namedef{subjclassname@2010}{Mathematical subject classification 2010}
\@namedef{subjclassname@2020}{Mathematical subject classification 2020}
\makeatother

\newtheorem{theorem}{Theorem}[section]
\newtheorem{lemma}[theorem]{Lemma}
\newtheorem{corollary}[theorem]{Corollary}
\newtheorem{proposition}[theorem]{Proposition}

\newtheorem{question}[theorem]{Question}
\newtheorem{problem}[theorem]{Problem}
\newtheorem{defn_thm}[theorem]{Definition/Theorem}

\theoremstyle{definition}
\newtheorem{definition}[theorem]{Definition}
\newtheorem{remark}[theorem]{Remark}

\theoremstyle{plain}

\newtheoremstyle{TheoremNum}
     {8.0pt plus 2.0pt minus 4.0pt}{8.0pt plus 2.0pt minus 4.0pt} 
     {\itshape} 
     {-0.15cm} 
     {\bfseries} 
     {.} 
     { }  
     {\thmname{#1}\thmnote{ \bfseries #3}}
    \theoremstyle{TheoremNum}
\newtheorem{duplicate}{}

\newcommand*{\claimproofname}{My proof}

\newcommand{\FP}{\mathsf{FP}}
\newcommand{\F}{\mathsf{F}}

\newcommand{\Z}{\mathbb{Z}}
\newcommand{\ZZ}{\mathbb{Z}}
\newcommand{\RR}{\mathbb{R}}
\newcommand{\NN}{\mathbb{N}}
\newcommand{\QQ}{\mathbb{Q}}

\DeclareMathOperator{\rank}{rank}
\DeclareMathOperator{\Hom}{Hom}

\usepackage{tikz}
\usetikzlibrary{arrows,quotes}
\tikzstyle{blackNode}=[fill=black, draw=black, shape=circle]

\def\raag{right-angled Artin group }
\def\raags{right-angled Artin groups }
\def\raagstop{right-angled Artin group. }

\newcommand{\lb}{b^{(2)}}
\newcommand{\calD}{\mathcal{D}}
\newcommand{\rankD}{\mathrm{rank}_{\calD_G}}

\DeclareMathOperator{\thi}{th}
\DeclareMathOperator{\fab}{fab}

\DeclareMathOperator{\cut}{cut}

\makeatletter
\DeclareFontFamily{OMX}{MnSymbolE}{}
\DeclareSymbolFont{MnLargeSymbols}{OMX}{MnSymbolE}{m}{n}
\SetSymbolFont{MnLargeSymbols}{bold}{OMX}{MnSymbolE}{b}{n}
\DeclareFontShape{OMX}{MnSymbolE}{m}{n}{
    <-6>  MnSymbolE5
   <6-7>  MnSymbolE6
   <7-8>  MnSymbolE7
   <8-9>  MnSymbolE8
   <9-10> MnSymbolE9
  <10-12> MnSymbolE10
  <12->   MnSymbolE12
}{}
\DeclareFontShape{OMX}{MnSymbolE}{b}{n}{
    <-6>  MnSymbolE-Bold5
   <6-7>  MnSymbolE-Bold6
   <7-8>  MnSymbolE-Bold7
   <8-9>  MnSymbolE-Bold8
   <9-10> MnSymbolE-Bold9
  <10-12> MnSymbolE-Bold10
  <12->   MnSymbolE-Bold12
}{}

\let\llangle\@undefined
\let\rrangle\@undefined
\DeclareMathDelimiter{\llangle}{\mathopen}%
                     {MnLargeSymbols}{'164}{MnLargeSymbols}{'164}
\DeclareMathDelimiter{\rrangle}{\mathclose}%
                     {MnLargeSymbols}{'171}{MnLargeSymbols}{'171}
\makeatother

\title[Thurston norm for RAAGs]{Thurston norm for coherent right-angled Artin groups via $L^2$-invariants}

\usepackage[foot]{amsaddr}
\author{Monika Kudlinska}
\address{DPMMS, Centre for Mathematical Sciences, Wilberforce Road, Cambridge,
CB3 0WB, UK, and Emmanuel College, St Andrew's Street, Cambridge CB2 3AP, UK}
\email{m.kudlinska@dpmms.cam.ac.uk}

\begin{document}

\maketitle

\begin{abstract}
   We define a new notion of splitting complexity for a group $G$ along a non-trivial integral character $\phi \in H^1(G; \Z)$. If $G$ is a one-ended coherent right-angled Artin group, we show that the splitting complexity along an epimorphism $\phi \colon G \to \Z$ equals the $L^2$-Euler characteristic of the kernel of $\phi$. This allows us to define a Thurston-type semi-norm $\| \cdot \|_T \colon H^1(G ; \RR) \to \RR$ which measures the splitting complexity of integral characters. 
   
   Our main tool is Friedl--L\"{u}ck's $L^2$-polytope which we compute for all groups in this class. We also compute the $L^2$-Betti numbers of all kernels of characters $G \to \Z$ and show that the $L^2$-homology of any finitely generated subgroup of $G$ is concentrated in the first dimension, which may be of independent interest.
\end{abstract}

\section{Introduction}

Let $M$ be a compact orientable 3-manifold. For any compact surface $S$ with connected components $S_1, \ldots, S_n$, the complexity of $S$ is defined to be
\[ \chi_{-}(S) := \sum_{i = 1}^{n} \max\{-\chi(S_i), 0\}.\]
The complexity of an integral character $\phi \in H^1(M; \ZZ)$ is the minimal complexity of a properly embedded surface $S \subset M$ that represents the second homology class of $M$ dual to $\phi$. 

In his seminal paper \cite{Thurston1986}, Thurston shows that the function $H^1(M ;\Z) \to \RR$ that assigns to each character $\phi$ its complexity can be extended continuously to a function $\| \cdot \|_T$ on $H^1(M ; \RR)$ that is convex and linear on rays through the origin. We call $\| \cdot \|_T$ the \emph{Thurston semi-norm} on $M$.
The unit ball of the Thurston semi-norm, 
\[B_T := \{\phi \in H^{1}(M ; \RR) : \| \phi \|_T \leq 1\},\]
is a finite-sided rational polytope which is referred to as the \emph{Thurston polytope}.


\medskip

In this paper we propose a group-theoretic analogue of the Thurston semi-norm which measures the \emph{splitting complexity} of integral characters of a group. More precisely, let $G$ be a group and let $\phi \in H^1(G; \Z)$ be a non-trivial character. A graph-of-groups splitting of $G$ is said to be \emph{dual} to $\phi$ if every vertex group is identified with a subgroup of the kernel of $\phi$. The \emph{complexity} of a graph-of-groups splitting, roughly, measures the total sum of the Euler characteristics of the edge groups. The \emph{splitting complexity} $c(G; \phi)$ of $G$ along $\phi$ is defined to be the infimal complexity over all graph-of-groups splittings of $G$ dual to $\phi$ with vertex and edge groups of type $\FP$.  If $M$ is a compact, connected, orientable, aspherical 3-manifold, and $\phi$ a non-trivial integral character, then the splitting complexity of $\pi_1(M)$ along $\phi$ coincides with the Thurston norm of $\phi$.

We begin our investigation of the group-theoretic Thurston semi-norm within the family of \emph{right-angled Artin groups}. Recall, a right-angled Artin group is any group which admits a presentation with finitely many generators, where the only relations are commutators between the defining generators. Such a presentation can be encoded via a simplicial graph, where the vertices correspond to the generators, and two generators commute whenever the corresponding vertices are connected by an edge in the graph. 

Right-angled Artin groups form a natural and widely studied class of groups. They have played a key role in the resolution of Thurston's virtual Haken conjecture, and they continue to be the focus of study, both in their own right and through their applications to Geometric Group Theory and Low-Dimensional Topology.

There is a close relationship between the algebra of a right-angled Artin group and the combinatorial structure of its defining graph and its flag completion (see \cite{Koberda2022} for a recent survey of this subject). In particular, a right-angled Artin group is coherent if and only if the defining graph has no induced $n$-cycles for $n \geq 4$ \cite{Droms1987}. A well-known family of coherent right-angled Artin groups consists of those of \emph{Droms-type}. There has been much recent progress in studying limit groups over coherent and Droms \raags \cite{CRMAK2023, KochloukovaLGZ2022, Fruchter2023}.


The purpose of this article is to show that coherent \raags admit a Thurston-type norm which is completely analogous to the Thurston norm in the setting of 3-manifolds:

\begin{duplicate}[\cref{thm_main}]
    Let $G$ be a one-ended coherent \raagstop There is a canonical continuous function 
    \[ \| \cdot \|_T \colon H^1(G; \RR) \to \RR \]
    which satisfies the following. 
    \begin{enumerate}
        \item The function $\| \cdot \|_T$ is convex and linear on rays through the origin. 
        \item For any non-trivial integral character $\phi \in H^1(G ; \Z)$, we have that $\| \phi \|_T = c(G; \phi)$ where $c(G; \phi)$ is the \emph{splitting complexity of $G$ along $\phi$} (see \cref{def:splitting_complexity}).
        \item In particular, for any algebraically fibered epimorphism $\phi \colon G \to \Z$, we have that $\| \phi \|_T = \chi(\ker \phi)$, where $\chi(\ker \phi)$ is the Euler characteristic of the group $\ker \phi$.
    \end{enumerate}
\end{duplicate}

Recall that an epimorphism $\phi \colon G \to \Z$ is said to be \emph{algebraically fibered} if $\ker \phi$ is finitely generated.

\begin{remark}We note that \cref{thm_main} is false without the assumption of coherence. Indeed, if $G$ is an incoherent one-ended right-angled Artin group then there exists an epimorphism $\phi \colon G \to \Z$ which has a finitely generated and non-finitely presented kernel. Then, the HNN splitting $G \cong \ker \phi \ast_{\ker \phi}$ is the unique minimal splitting dual to $\phi$. Hence, there does not exist a splitting dual to $\phi$ that is of finite type.
\end{remark}

\subsection{Proof strategy via $L^2$-polytopes}

Our approach for proving \cref{thm_main} utilises Friedl--L\"{u}ck's \emph{universal $L^2$-torsion}, which is an invariant of certain $L^2$-acyclic $\Z G$-chain complexes first defined in \cite{FriedlLueck2017}. The universal $L^2$-torsion encompasses a number of important invariants from the world of 3-Manifold Theory and Knot Theory, including the $L^2$-torsion and the Thurston polytope. In particular, it gives rise to the so-called \emph{$L^2$-polytope}, which coincides with the Thurston polytope in the setting of compact $L^2$-acyclic 3-manifolds (see, for example, \cite[pp.72--73]{DeBloisFriedlVidussi2014}).

Let $G$ be an $L^2$-acyclic group of type $\F$ that satisfies the Atiyah conjecture. Let $X$ be a free finite $G$-CW complex. The \emph{universal $L^2$-torsion} of $X$, denoted by $\rho^{(2)}_u(X; G)$, is an element of the weak Whitehead group $\mathrm{Wh}^w(G)$. It is computed via a weak chain contraction of the $L^2$-acyclic chain complex associated to the $G$-action on $X$. The \emph{$L^2$-polytope} $\mathcal{P}(X; G)$ is the image of the universal $L^2$-torsion $ - \rho^{(2)}_u(X; G)$ under the polytope homomorphism,
\[ \mathcal{P}_T \colon \mathrm{Wh}^w(G) \to \mathfrak{P}_T(G_{\mathrm{fab}}),\]
where $\mathfrak{P}_T(G_{\mathrm{fab}})$ denotes the Grothendieck group of translation-invariant polytopes in $\RR \otimes_{\Z} G_{\mathrm{fab}}$. Kielak shows that the $L^2$-polytope is a group invariant in the setting of $L^2$-acyclic groups of type $\F$ which satisfy the Atiyah conjecture \cite[Theorem~5.17]{Kielak2020}. 

The $L^2$-polytope $\mathcal{P}(G)$ gives rise to a map 
\[\mathrm{th}_{\bullet}\mathcal{P}(G) \colon H^1(G; \RR) \to \RR,\]
which measure the thickness of each character $\phi \in H^1(G; \RR)$ along the polytope $\mathcal{P}(G)$.
In the case that the $L^2$-polytope is represented by a \emph{single polytope}, the thickness map is a semi-norm on the vector space $H^1(G; \RR)$. Moreover, the thickness of a non-trivial integral character $\phi \colon G \to \Z$ along the $L^2$-polytope is exactly the $L^2$-Euler characteristic of its kernel, 
\[\mathrm{th}_{\phi}\mathcal{P}(G) = - \chi^{(2)}(\ker \phi). \]

\medskip

The key technical theorem in this paper involves relating the $L^2$-Euler characteristic of the kernel of an epimorphism $\phi \colon G \to \Z$ to the splitting complexity of the group $G$ along $\phi$:

\begin{theorem}\label{thm:L2Euler}
    Let $G \not\cong \Z$ be a one-ended coherent right-angled Artin group and let $\phi \colon G \to \Z$ be an epimorphism. Then, 
    \[ - \chi^{(2)}(\ker \phi) = c(G; \phi),\]
    where $c(G; \phi)$ is the splitting complexity of $G$ along $\phi$.
\end{theorem} 
\medskip

A 3-manifold $M$ is said to be \emph{admissible} if it is connected, compact, orientable, irreducible with empty or toroidal boundary and $\pi_1(M)$ is infinite. We compare \cref{thm:L2Euler} with the following result:

\begin{theorem}[{Friedl--L\"{u}ck \cite[Theorem~0.2]{FriedlLueck2019}}]
    Let $M \neq S^2 \times D^1$ be an admissible 3-manifold and let $\phi \colon \pi_1(M) \to \Z$ be an epimorphism. Then,
    \[ -\chi^{(2)}(\ker \phi) = c(\phi),\]
    where $c(\phi)$ denotes the minimal complexity of an embedded oriented surface $S$ which represents the class $c \in H^1(M, \partial M ; \Z)$ which is the Poincar\'{e}-Lefschetz dual of $\phi$.
\end{theorem}

\subsection{Extensions to other classes of groups}

As discussed above, if $G$ is incoherent and if there exists an epimorphism $\phi \colon G \to \Z$ whose kernel is a witness to incoherence then $c(G; \phi) = \infty$. We say a group $G$ is \emph{strongly coherent} if any finitely generated subgroup is of type $\FP$.

Thus, a natural question arises:

\begin{problem}\label{problem}
    Determine the classes of strongly coherent groups for which the function
    \[ \begin{split} H^1(G; \Z) \to \RR \\
    \phi \mapsto c(G; \phi) 
    \end{split}\]
    extends to a semi-norm on $H^1(G; \RR)$. Here, $c(G; \phi)$ denotes the splitting complexity of $G$ along $\phi$.
\end{problem}

We will now briefly discuss some classes of groups for which progress on \cref{problem} has been made.

Let $G$ be a torsion-free 2-generator 1-relator group such that $G \not\cong \mathbb{F}_2$ By the recent work of Jaikin--Linton, all such groups are strongly coherent \cite{JaikinLinton2023}. Moreover, any such group is $L^2$-acyclic, of type $\F$ and satisfies the Atiyah conjecture by \cite{JaikinZapirainLopezAlvarez2020}. Thus, the $L^2$-polytope $\mathcal{P}(G)$ is defined and is an invariant of the group. Moreover, it can be explicitly calculated from a \emph{nice} 2-generator 1-relator presentation of $G$ (see \cite{FriedlTilmann2020} and the discussion in Section~3.2 of \cite{FriedlLuckTillmann2019}). In particular, the $L^2$-polytope $\mathcal{P}(G)$ is a single polytope. 

\begin{theorem}[Friedl--L\"{u}ck--Tillmann {\cite[Theorem~5.2]{FriedlLuckTillmann2019}},  Henneke--Kielak {\cite[Theorem~6.4]{HennekeKielak2020}}]
    Let $G$ be a torsion-free 2-generator 1-relator group such that $G \not \cong \mathbb{F}_2$. Let $\phi \colon G \to \Z$ be an epimorphism. Then, 
    \[c(G; \phi) = - \chi^{(2)}(\ker \phi) = \mathrm{th}_{\phi} \mathcal{P}(G).\]
\end{theorem}

\begin{corollary}
    If $G$ is a torsion-free 2-generator 1-relator group such that $G \not \cong \mathbb{F}_2$ then the conclusions of \cref{thm_main} hold for $G$.
\end{corollary}

Finally, let us briefly discuss the class of \emph{free-by-cyclic} groups. We say that a group $G$ is free-by-cyclic if it contains a normal subgroup $N$ which is isomorphic to a free group of finite rank and such that $G / N \cong \Z$. A free-by-cyclic group admits a finite classifying space given by the mapping torus of a realisation of the monodromy map. Free-by-cyclic groups are $L^2$-acyclic \cite[Theorem~1.39]{Lueck2002} and satisfy the Atiyah conjecture \cite[Theorem~1.3]{Linnell1993}. Thus, the $L^2$-polytope of $G$ is well defined and is a single polytope by the proof of \cite[Theorem~5.29]{Kielak2020} (in fact, it was already shown in \cite[Corollary~3.5]{FunkeKielak2018} that the thickness function $\mathrm{th}_{\bullet}\mathcal{P}(G) \colon H^1(G; \RR) \to \RR$ is a semi-norm). 

Hence, in order to prove that free-by-cyclic groups admit a Thurston semi-norm, it suffices to prove the following:

\begin{problem}
    Let $G$ be a free-by-cyclic group and let $\phi \colon  G\to \Z$ be an epimorphism. Show that 
     \[ -\chi^{(2)}(\ker \phi) = c(G; \phi).\]
\end{problem} A variation of this problem has been posed by Gardam--Kielak in \cite{GardamKielak, Oberwolfach}. See also \cite{KielakSun2024} for related work.

\subsection*{Acknowledgments}

The author would like to thank Dawid Kielak for introducing her to the work of Friedl--L\"{u}ck and Friedl--Tillmann, as well as his own work on polytopes which inspired this paper. She would also like to thank Dawid Kielak for comments on an earlier draft. Finally, the author thanks the anonymous referee for useful comments. 

\section{Preliminaries}

\subsection{Ring theory}

Every ring in this article is assumed to be associative and unital.

\begin{definition}[Ore condition] \label{definition_Ore}
    Let $R$ be a ring and $S \subseteq R$ a multiplicative subset. The ring $R$ satisfies the \emph{right Ore condition for $S$} if for any $r \in R$ and $s \in S$ there exists some $u \in R$ and $v \in S$ such that $rv = su$. A ring $R$ is a \emph{right Ore domain} if $R$ is a domain and satisfies the right Ore condition with respect to $R - \{0\}$.
\end{definition}

Suppose that $R$ satisfies the right Ore condition with respect to $S$. We define an equivalence relation $\sim$ on the Cartesian product $R \times S$ so that $(r,s) \sim (r',s')$ if and only if there exist elements $t, t' \in R$ such that $rt = r't'$ and $st = s't'$. The \emph{Ore localisation of $R$ with respect to $S$} is the set of equivalence classes \[\mathrm{Ore}(R,S) := \left.(R \times S) \middle/ \sim \right. .\]
When $R$ is an Ore domain and $S = R - \{0\}$, we simply write \[\mathrm{Ore}(R) := \mathrm{Ore}(R, R - \{0\}) .\]

\begin{definition}
    A ring $R$ is a \emph{division ring} if for any $r \in R - \{0\}$ there exists $s \in R$ such that $sr = 1$ and $rs = 1$. Note that if such an element $s$ exists then it is unique. We call such $s$ the \emph{inverse} of $r$ and denote it by the symbol $r^{-1}$.
\end{definition}

If $R$ is a division ring and $S$ is a subset of $R$, then we define the \emph{division closure} of $S$ in $R$, denoted by $\bar{S}$, to be the intersection of all division closed subrings of $R$ that contain $S$. 

\begin{lemma}\label{lemma:oreDomainSkew}
    If $R$ is an Ore domain then the Ore localisation $\mathrm{Ore}(R)$ is a division ring.
\end{lemma}



Let $\pi \colon G \twoheadrightarrow H$ be a group epimorphism with kernel $K$. Fix a set-theoretic section $s \colon H \to G$ of $\pi$. Let $\Z K \ast H$ be the set of all finite linear combinations $\sum_{h\in H} \lambda_h\ \ast h$ of elements in $H$ with coefficients in the group ring $\Z K$.

We define a ring structure on the set $\Z K \ast H$, as follows. 
Addition is defined to be pointwise. Multiplication in $\Z K \ast H$ is defined by linearly extending the operation 
\[ (\lambda_g \ast g) \cdot (\lambda_h \ast h) := \lambda_g \mathrm{ad}_{s(g)}(\lambda_h) s(g) s(h) s(gh)^{-1} \ast gh,\]
where for any $x \in G$, \[\mathrm{ad}_x \colon \Z K \to \Z K\]
is the function defined by linearly extending the conjugation action by $x$ on $K$, $k \mapsto x k x^{-1}$ for all $k \in K$.

 We write $\Z K \ast_s H$ to denote the resulting ring. 

\begin{lemma}\label{lemma:twisted_group_ring}
    The twisted group ring $\Z K \ast_s H$ is isomorphic to the (untwisted) group ring $\Z G$ via the map $\Z G \to \Z K \ast_s H$ defined by linearly extending
    \[ g \mapsto g \cdot ((s \circ \pi ) (g))^{-1} \ast \pi(g). \]
\end{lemma}

\subsection{Bass--Serre theory}

A \emph{(combinatorial) graph} $X$ consists of a tuple of sets $(V(X), E(X))$, a fixed point free involution $ - \colon E(X) \to E(X)$, and a pair of maps 
\[i, \tau \colon\ E(X) \to V(X),\]
such that $i(\bar{e}) = \tau(e)$ for every $e \in E$. An \emph{orientation} $E^{+}(X)$ of $X$ is a subset of edges $E^{+}(X) \subseteq E(X)$ such that the sets $E^{+}(X)$ and $\widebar{E^{+}(X)}$ together form a partition of the edge set $E(X)$.

A \emph{star graph} $S_n$ is the complete bipartite graph $K_{1,n}$. A \emph{rose graph} $R_n$ is the graph with a single vertex and $2n$ edges.
\medskip

A \emph{graph of groups} $\mathbb{X}$ consists of a graph $X$ and an assignment of a group $\mathbb{X}_v$ to every vertex $v \in V(X)$ and a group $\mathbb{X}_e$ to every edge $ e\in E(X)$ so that $\mathbb{X}_e = \mathbb{X}_{\bar{e}}$. For each edge $e \in E(X)$, there is a monomorphism $\iota_e \colon \mathbb{X}_e \hookrightarrow \mathbb{X}_{\tau(e)}$.

Let $\mathbb{X} = (X, \mathbb{X}_{\bullet}, \iota_{\bullet})$ be a graph of groups. Let $F(E(X))$ denote the free group with a free basis identified with the edges of $X$ and define $F_{\mathbb{X}}$ to be the free product\[ F_{\mathbb{X}} = \bigast_{v\in V(X)} \mathbb{X}_v \bigast F(E(X)). \]
Let $N \trianglelefteq F_{\mathbb{X}}$ be the normal subgroup of $F_{\mathbb{X}}$ given by the normal closure 
\[N = \llangle \{ e \bar{e}, \forall e \in E(X)) \} \cup \{ e\iota_e(a)e^{-1}(\iota_{\bar{e}}(a))^{-1}, \forall e\in E(X),\forall a\in \mathbb{X}_e  \}\rrangle.  \]
Finally, let us fix a spanning tree $\mathcal{T}$ of the graph $X$. Then the \emph{fundamental group of the graph of groups $\mathbb{X}$} with respect to the spanning tree $\mathcal{T}$ is the quotient group 
\[ \pi_1(\mathbb{X}, \mathcal{T}) =  \left(F_{\mathbb{X}} / N \right) / \llangle \{ e \in E(\mathcal{T})\} \rrangle. \]
Note that for any two different choices of spanning trees $\mathcal{T}$ and $\mathcal{T}'$ of the graph $X$, there exists an isomorphism between the corresponding fundamental groups of graphs of groups, \[\pi_1(\mathbb{X}, \mathcal{T}) \cong \pi_1(\mathbb{X}, \mathcal{T}').\] 

We say that the group $G$ admits a \emph{graph-of-groups splitting} $\mathbb{X}$ if there is an isomorphism
\[ G \cong \pi_1(\mathbb{X}, \mathcal{T}),\]
for some spanning tree $\mathcal{T}$.

A \emph{$G$-tree $T$} is a simplicial tree $T$ with an action \[ \begin{split}G &\to \mathrm{Aut}(T)  \\ g &\mapsto \sigma_g \end{split}\] which is \emph{without inversion}, meaning that $\sigma_g(e) \neq \bar{e}$ for every $e\in E(X)$.  We will often abuse notation by using the symbol $g$ to denote the automorphism $\sigma_g$ of $T$ induced by the element $g \in G$.

The \emph{quotient graph of groups} $\mathbb{X}_T$ for the $G$-tree $T$ is a graph of groups where the underlying graph $X$ is the quotient of $T$ by the $G$-action, 
\[X:= G  \backslash T.\] 
We pick a spanning tree $\mathcal{T} \subseteq X$ and a connected lift $\hat{\mathcal{T}}$ of $\mathcal{T}$ in $T$ (we call such $\hat{\mathcal{T}}$ the \emph{tree of representatives} of $T$ modulo $G$). The vertex and edge groups of the quotient graph of groups are identified with the stabilisers of the lifts of the corresponding vertices and edges of $\hat{\mathcal{T}}$. The edge inclusion maps are induced by subgroup inclusion. Let $e \in E(X)$ be an edge which is not contained in the spanning tree $\mathcal{T}$. Let $\hat{e}$ be a lift of $e$ to $T$ such that $\iota(\hat{e}) \in V(\hat{\mathcal{T}})$. We let $t_e \in G$ denote the element of $G$ such that $t_e \cdot \hat{e} \in \hat{\mathcal{T}}$. Then, there is a homomorphism \[\Phi \colon \pi_1(\mathbb{X}_T, \mathcal{T}) \to G\] which sends the vertex and edge groups of $\mathbb{X}_T$ to the corresponding subgroups of $G$, and the edges $e \in E(X \setminus T)$ to the elements $t_e$. It is a standard result from Bass--Serre Theory that $\Phi$ is an isomorphism. Moreover, for any graph-of-groups splitting 
\[ G \cong \pi_1(\mathbb{X}, \mathcal{T}),\]
there exists a $G$-tree $T$ known as the \emph{Bass--Serre tree} of the splittings, such that the quotient graph of groups is isomorphic to $\mathbb{X}$. See \cite{Serre2003} for more details.
\subsection{Finiteness properties}

Let $G$ be a group and $R$ a ring. We say that $G$ is of \emph{type $\FP$ over $R$}, denoted by $\FP (R)$, if there exists a finite resolution of the trivial $RG$-module $R$ by finitely generated $RG$-modules. A group is of \emph{type $\FP$} if it is of type $\FP(\Z)$.

\subsection{$L^2$-Betti numbers}

Let $L^2(G)$ denote the Hilbert space of square-summable formal sums over $G$ with coefficients in $\mathbb{C}$, 
\[L^2(G) = \left\{\sum_{g \in G} \lambda_g g, \lambda_g \in \mathbb{C} : \sum |\lambda_g|^2 < \infty \right\}.\]
The \emph{group von Neumann algebra} $\mathcal{N}(G)$ of $G$ is the algebra of $G$-equivariant bounded operators from $L^2(G)$ to $L^2(G)$. It has a natural $\mathcal{N}(G)$-$\mathbb{Z}G$-bimodule structure.

Let $X$ be a $G$-CW-complex and let $C_{\bullet}(X)$ be the associated cellular $\Z G$-chain complex. Define the $L^2$-homology of the $G$-CW-complex $X$ to be \[H_p^{(2)}(X; G) := H_p(  \mathcal{N}(G) \otimes_{\Z G} C_{\bullet}(X)) = \ker \left. \partial_p \big/\, \mathrm{cl}\left(\mathrm{Im}\,\partial_{p+1}\right) \right. ,\]
where $\mathrm{cl}(\mathrm{Im}\,\partial_{p+1})$ denotes the closure of $\mathrm{Im}\,\partial_{p+1}$ in the Hilbert space $\ker \partial_p$.

The $p$-th $L^2$-Betti number $b_p^{(2)}(X; G)$ of $X$ is defined to be 
\[b_p^{(2)}(X; G) := \mathrm{dim}_{\mathcal{N}(G)}(H_p^{(2)}(X;G)) \in [0, \infty].\]
Here, $\mathrm{dim}_{\mathcal{N}(G)}(\cdot)$ denotes the \emph{extended von Neumann dimension} of an $\mathcal{N}(G)$-module in the sense of L\"{u}ck (see \cite[Definition~6.20]{Lueck2002}). Note that, in particular, we do not require the $G$-CW complex $X$ to be cocompact.

The $p$-th $L^2$-Betti number of a group $G$ is defined to be 
\[b_p^{(2)}(G) := b_p^{(2)}(\tilde{X}; G)\]
where $\tilde{X}$ is the universal cover of a classifying space $X$ for $G$. A group $G$ is said to be \emph{$L^2$-acyclic} if $b_p^{(2)}(G) = 0$ for all $p \geq 0$.

\medskip

A group $G$ is said to be \emph{$L^2$-finite} if
\[ \sum_{p \geq 0} b_{p}^{(2)}(G) < \infty.\] Then, the \emph{$L^2$-Euler characteristic} of $G$ is defined to be
\[ \chi^{(2)}(G) = \sum_{p \geq 0} (-1)^p \cdot b_p^{(2)}(G).\]
\begin{lemma}[{\cite[Theorem 1.35(2)]{Lueck2002}}]\label{lemma:type_F_euler}
     If $G$ is of type $\F$ then 
\[\chi^{(2)}(G) = \chi(G).\]
\end{lemma}

\subsubsection{The Atiyah conjecture}

The elements of $G$ act as bounded operators $L^2(G) \to L^2(G)$ and hence there is a natural embedding $\mathbb{Q} G \hookrightarrow \mathcal{N}(G)$. We define the \emph{Linnell ring} $\mathcal{D}_G$ of a group $G$ to be the division closure of $\mathbb{Q}G$ in the Ore localisation $\mathrm{Ore}(\mathcal{N}(G), S)$, where $S$ is the set of elements in $\mathcal{N}(G)$ which are not zero divisors. We note that the group von Neumann algebra $\mathcal{N}(G)$ satisfies the Ore condition with respect to its set of non-zero-divisors by \cite[Theorem~8.22]{Lueck2002}.

The \emph{Atiyah conjecture} predicts the rationality of the von Neumann dimension of the kernel of any bounded operator $L^2(G)^n \to L^2(G)^m$. For the sake of brevity, we will use the following characterisation which follows from the work of Linnell:

\begin{defn_thm}[Linnell \cite{Linnell1993}]\label{def:Atiyah}
    A torsion-free group $G$ satisfies the \emph{Atiyah conjecture over $\mathbb{Q}$} if the Linnell ring $\mathcal{D}_G$ is a division ring. 
\end{defn_thm}

\begin{remark}
Some places in the literature refer to the statement of \cref{def:Atiyah} as the \emph{strong} Atiyah conjecture.
\end{remark}

\begin{theorem}[{L\"{u}ck \cite[Lemma 10.28(3)]{Lueck2002}}]\label{lueck}
    Let $G$ be a torsion-free group which satisfies the Atiyah conjecture. Then, for any $p \geq 0$ the $p$-th $L^2$-Betti number of $G$ is given by
    \[b_p^{(2)}(G) = \mathrm{rank}_{\mathcal{D}_G}H_p(G; \mathcal{D}_G).\]
\end{theorem}

\section{Polytopes}

\subsection{The polytope group}

Let $H$ be a free abelian group of finite rank $n$. Define $V := H \otimes_{\Z} \RR$ and identify $H$ with the standard integer lattice in $V$. A \emph{polytope} $P$ in $V$ is defined to be the convex hull of finitely many points in $V$. A polytope is said to be \emph{integral} if its vertices are contained in the standard integer lattice. 

Let $P \subseteq V$ be a polytope and $\phi \colon V \to \RR$ a linear map. The \emph{thickness of $P$ along $\phi$} is 
\[\thi_{\phi}(P) := \sup\{|\phi(p) - \phi(q)| : p,q \in P\}.\]

\medskip

The \emph{Minkowski sum} of two polytopes $P,Q \subseteq V$ is defined to be the polytope
\[P + Q := \{p + q \mid p\in P, q\in Q\}.\]
The set of non-empty integral polytopes forms an abelian monoid with respect to the Minkowski sum. We define an equivalence relation on pairs of non-empty integral polytopes such that  $(P_1, Q_1) \sim (P_2, Q_2)$ whenever $P_1 + Q_2 = P_2 + Q_1$. Let $\mathfrak{P}(H)$ be the resulting set of equivalence classes. Then $\mathfrak{P}(H)$ is an abelian group with respect to the operation + defined as
\[[(P_1, Q_1)] + [(P_2, Q_2)] := [(P_1 + P_2, Q_1 + Q_2)].\] We call $\mathfrak{P}(H)$ the \emph{Grothendieck group of polytopes} in $H \otimes_{\Z} \RR$, or the \emph{polytope group}, for short. An element of $\mathfrak{P}(H)$ is said to be a \emph{single polytope} if it is represented by a pair $(P,Q)$ where $Q$ is a singleton.

Let $\mathfrak{C}(H)$ denote the subgroup of $\mathfrak{P}(H)$ which consists of elements represented by pairs $(P,Q)$ where $P$ and $Q$ are both singletons. We define the \emph{translation-invariant polytope group} $\mathfrak{P}_T(H)$ to be the quotient
\[ \mathfrak{P}_T(H) := \left.\mathfrak{P}(H) \middle/ \mathfrak{C}(H)\right. \]
An element of $\mathfrak{P}_T(H)$ is said to be a \emph{single polytope} if it admits a single polytope representative in $\mathfrak{P}(H)$.

\medskip

Let $G$ be a finitely generated group. We define $\mathfrak{P}_T(G)$ to be the translation-invariant polytope group $\mathfrak{P}_T(G_{\fab})$ where $G_{\fab}$ is the free abelianisation of $G$. The assignment $G \mapsto \mathfrak{P}_T(G)$ defines a covariant functor from the category of finitely generated groups to the category of abelian groups. Given a group homomorphism $\psi \colon G_1 \to G_2$, we write $\psi_{*} \colon \mathfrak{P}_T(G_1) \to \mathfrak{P}_T(G_2)$ to denote the induced homomorphism of translation-invariant polytope groups.

For any element $\mathcal{P} \in \mathfrak{P}_T(H)$ which is represented by a tuple $(P,Q)$ and any linear map $\phi \colon V \to \RR$, the \emph{thickness of $\mathcal{P}$ along $\phi$} is defined as 
\[\thi_{\phi}(\mathcal{P}) := \thi_{\phi}(P) - \thi_{\phi}(Q).\]
It is easy to check that this is well-defined and gives rise to a homomorphism
\[\thi_{\phi} \colon \mathfrak{P}_T(H) \to \RR\]

The following lemma is elementary.

\begin{lemma}\label{lemma:polytope_semi-norm}
    Let $H$ be a finite rank free abelian group and let $\mathcal{P} \in \mathfrak{P}_T(H)$ be a single polytope. 
    Then the map 
    \[\begin{split}
        \thi_{\bullet}(\mathcal{P}) \colon \mathrm{Hom}(H; \RR) &\to \RR \\
        \phi \mapsto \thi_{\phi}(\mathcal{P})
    \end{split}\]
    defines a semi-norm on $\mathrm{Hom}(H; \RR)$.
\end{lemma}




\subsection{The polytope homomorphism}

As before, let $H$ be a finite rank free abelian group and let $V = H \otimes_{\Z} \RR$. Suppose that $R$ is a ring and let $RH$ be a (possibly twisted) group ring with no zero divisors. Let $x\in RH$ be a non-zero element. Then $x$ is a formal sum
\[x = \sum_{h\in H} x_h h\]
with at most finitely many $x_h\in R$ non-zero. 

The $\mathcal{P}(x)$ of an element $x \in RH$ is defined to be the convex hull 
\[\mathcal{P}(x) := \mathrm{hull}\{h \mid x_h \neq 0\} \subseteq V.\]
We define $\mathcal{P}(0)$ to be the empty set.

Suppose that $RH$ is an Ore domain and let $\mathrm{Ore}(RH)$ be its Ore localisation. Then, for any $z \in \mathrm{Ore}(RH)$, we can realise $z$ as $z = xy^{-1}$ for some $x, y \in R H$. We define 
\[\mathcal{P}(z) = \mathcal{P}(xy^{-1}) := \mathcal{P}(x) - \mathcal{P}(y) \in \mathfrak{P}(H).\] Note that since $RH$ is an Ore domain, the Ore localisation $\mathrm{Ore}(RH)$ is a division ring \cite[Corollary~1.3.3]{Cohn1977}.

\begin{lemma}[{\cite[Lemma~3.12]{Kielak2020}}]\label{lemma_polytope_hom}
The map $\mathcal{P}$ is well-defined and induces a homomorphism
    \[\mathcal{P}: \mathrm{Ore}(RH)^{\times}_{\mathrm{ab}} \to \mathfrak{P}(H).\]
\end{lemma}

\begin{remark}The homomorphism in \cref{lemma_polytope_hom} descends to a homomorphism of the quotient of $\mathrm{Ore}(RH)^{\times}_{\mathrm{ab}}$ by the subgroup $\{\pm h \mid h\in H\}$, with image in the translation-invariant polytope group, 
\[\mathcal{P}_T: \mathrm{Ore}(RH)^{\times}_{\mathrm{ab}} / \{\pm h \mid h\in H\} \to \mathfrak{P}_T(H).\] 
\end{remark}We will use the term \emph{polytope homomorphism} to describe both maps.

\subsection{Universal $L^2$-torsion}

Let $G$ be a group and let $V$ and $W$ be finite Hilbert $\mathcal{N}(G)$-modules. We say that a homomorphism $\rho \colon V \to W$ is a \emph{weak isomorphism} if it is injective and has dense image.

For any homomorphism $f \colon \Z G^k \to \Z G^l$, we write $\Lambda(f)$ to denote the homomorphism 
\[\Lambda(f) := f \otimes_{\Z G} \mathrm{id}_{\mathcal{N}(G)}  \colon \Z G^k \otimes_{\Z G} \mathcal{N}(G) \to \Z G^l \otimes_{\Z G} \mathcal{N}(G).\]

\begin{definition}[The weak $K_1$-group]
    The \emph{weak $K_1$-group} $K_1^w(\Z G)$ is the abelian group generated by endomorphisms $f \colon \Z G^k \to \Z G^k$ for all integers $k \geq 0$, such that $\Lambda(f)$ is a weak isomorphism of finite Hilbert $\mathcal{N}(G)$-modules, subject to the usual $K^1$-group relations. We write $[f]$ to denote the element of $K_1^w(\Z G)$ represented by the endomorphism $f$.

    We define the group $\widetilde{K}_1^w(\Z G)$ to be the quotient of $K_1^w(\Z G)$ by the subgroup $\{ \pm [\mathrm{id}] \}$, where $\mathrm{id}: \Z G \to \Z G$ is the identity map. The \emph{weak Whitehead group} $\mathrm{Wh}^w(G)$ of $G$ is the quotient of $K_1^w(\Z G)$ by the subgroup $\{ \pm [g] : g\in G \}$, where we write $g$ to denote the isomorphism $\Z G \to \Z G$ determined by right multiplication by the group element $g$.
\end{definition}

Let $(C_{\bullet}, \partial_{\bullet})$ be a chain complex of free finitely generated $\Z G$-modules with a fixed choice of unordered basis. From now on, we will call such a chain complex \emph{finite based free}. We write \[C_{\mathrm{odd}} := \oplus_{n\in \Z} C_{2n+1} \text{ and }C_{\mathrm{even}} := \oplus_{n \in \Z} C_{2n},\]
where $C_{\mathrm{even}}$ and $C_{\mathrm{odd}}$ admit the natural bases induced by the bases of the modules in the chain complex $C_{\bullet}$. 

Recall that a chain contraction $\gamma_{\bullet}$ of $(C_{\bullet}, \partial_{\bullet})$ is a sequence of maps $\gamma_n \colon C_n \to C_{n+1}$ for all $n \in \Z$, such that 
\[ \partial_{n+1}\circ\gamma_n + \gamma_{n-1} \circ \partial_n = \mathrm{id}_{C_n}.\] 

\begin{definition} Suppose that $(C_{\bullet}, \partial_{\bullet})$ is a chain complex of $\Z G$-modules. A \emph{weak chain contraction} of $(C_{\bullet}, \partial_{\bullet})$ is a pair of maps $(\gamma_{\bullet}, u_{\bullet})$, such that the following hold.
\begin{enumerate}
    \item The map $u_{\bullet} \colon C_{\bullet} \to C_{\bullet}$ is a chain map where each induced morphism of finite Hilbert modules \[\Lambda(u_n) \colon C_n \otimes_{\Z G} \mathcal{N}(G) \to C_n \otimes_{\Z G} \mathcal{N}(G)\] is a weak isomorphism.
    \item The map $\gamma_{\bullet} \colon C_{\bullet} \to C_{\bullet +1}$ is a null homotopy of $u_{\bullet}$.
    \item We have that $\gamma_n u_n = u_{n+1} \gamma_n$ for all $n \geq 0$.
\end{enumerate}
\end{definition}

By \cite[Lemma~2.5]{FriedlLueck2019}, any finite based free $L^2$-acyclic $\Z G$-chain complex admits a weak contraction.

We write $u_{\mathrm{odd}}$ to denote the map $C_{\mathrm{odd}} \to C_{\mathrm{odd}}$ defined by applying $u$ at each coordinate. We define the map 
\[(u\circ \partial + \gamma)_{\mathrm{odd}} \colon C_{\mathrm{odd}} \to C_{\mathrm{even}}\]
similarly.

\begin{definition}[Universal $L^2$-torsion]
    Let $(C_{\bullet}, \partial_{\bullet})$ be a finite based free $L^2$-acyclic $\Z G$-chain complex and let $(\gamma_{\bullet}, u_{\bullet})$ denote a weak chain contraction of $(C_{\bullet}, \partial_{\bullet})$. The \emph{universal $L^2$-torsion} of $(C_{\bullet}, \partial_{\bullet})$ is an element of $\widetilde{K}_1^w(\Z G)$ given by 
    \[ \rho^{(2)}_u(C_{\bullet}) = [(u\circ \partial + \gamma)_{\mathrm{odd}}] - [u_{\mathrm{odd}}].\]
    If $X$ is a finite connected $G$-CW-complex, let $C_{\bullet}(X)$ be the finite free chain complex associated to the $G$-CW structure on $X$. A choice of an element in each orbit of $n$-cells gives rise to an unordered basis of $C_{n}(X)$ for each $n \geq 0$. We define the universal $L^2$-torsion $\rho^{(2)}_u(X)$ to be the image of $\rho^{(2)}(C_{\bullet}(X))$ under the quotient
    \[ \widetilde{K}_1^w(\Z G) \twoheadrightarrow \widetilde{K}_1^w(\Z G) / \{ \pm [g] : g\in G \} = \mathrm{Wh}^w(G).\]
     
\end{definition}

\begin{remark} Recall that the Linell ring $\mathcal{D}_G$ is the division closure $\QQ G$ in the Ore localisation of the von Neumann algebra $\mathcal{N}(G)$,  with respect to its set of non-zero-divisors. 
If $G$ is a torsion free-group which satisfies the Atiyah conjecture, then $\mathcal{D}_G$ is a division ring by \cref{def:Atiyah}. In that case, the elements of the weak $K^1$-group $K_1^w(\Z G)$ are represented by endomorphisms of the form $f \colon \Z G^k \to \Z G^k$ where $f$ becomes an isomorphism upon tensoring by the identity $\mathrm{id}_{\mathcal{D}_G} \colon \mathcal{D}_G \to \mathcal{D}_G$, 
\[f \otimes_{\Z G} \mathrm{id}_{\mathcal{D}_G} \colon \Z G^k \otimes_{\Z G} \mathcal{D}_G \to \Z G^k \otimes_{\Z G} \mathcal{D}_G. \]
\end{remark}

\subsection{The $L^2$-polytope}

Let $G$ be a finitely generated group. Let $G_{\mathrm{fab}}$ denote the free abelianisation of $G$ and let $K$ be the kernel of the natural quotient map $\pi \colon G \twoheadrightarrow G_{\mathrm{fab}}$. Fix a set-theoretic section $s \colon G_{\mathrm{fab}} \to G$ of $\pi$.  By \cref{lemma:twisted_group_ring}, there exists an isomorphism between the group ring $\Z G$ and the twisted group ring $\Z K \ast_s G_{\mathrm{fab}}$.

Suppose moreover that $G$ is torsion free and satisfies the Atiyah conjecture. Then,  every automorphism of $K \trianglelefteq G$ extends to an automorphism of $\mathcal{D}_K$, and the embedding $( \mathbb{Z}K \hookrightarrow \mathcal{D}_K) \otimes_{K} \mathbb{Z}G$ gives rise to the twisted group ring $\mathcal{D}_K \ast_s G_{\mathrm{fab}}.$ 

Since $G_{\mathrm{fab}}$ is free abelian of finite rank and $K$ is torsion free and satisfies the Atiyah conjecture, by \cite[Theorem~10.38, Lemma~10.69]{Lueck2002} the twisted group ring $\mathcal{D}_K \ast_s G_{\mathrm{fab}}$ is an Ore domain, and the inclusion $\mathcal{D}_K \ast_s G_{\mathrm{fab}} \hookrightarrow \mathcal{D}_G$ induces an isomorphism
\[ F \colon \mathrm{Ore}(\mathcal{D}_K \ast_s G_{\mathrm{fab}}) \to \mathcal{D}_G.\]

\medskip

\begin{definition}[The $L^2$-polytope] \label{definition:l2_polytope}
    Let $G$ be a group of type $\F$ which satisfies the Atiyah conjecture. Let $Y$ be a finite $L^2$-acyclic $K(G,1)$-complex with universal cover $X$. Let $A$ be a matrix representative of $-\rho^{u}(X)$ where $\rho^{u}(X) \in \mathrm{Wh}^{w}(\Z G)$ is the universal $L^2$-torsion of $X$. The \emph{$L^2$-polytope} of $X$ is defined to be
    \[\mathcal{P}(X; G) :=  \mathcal{P}_T(F^{-1}(\mathrm{det}_{\mathcal{D}_G}(A \otimes \mathcal{D}_G))) \in \mathfrak{P}_T(G_{\mathrm{fab}}),\]
    where \[\mathrm{det}_{\mathcal{D}_G}(A \otimes \mathcal{D}_G) \in \mathcal{D}_G^{\times} /\, [\mathcal{D}_G^{\times}, \mathcal{D}_G^{\times}]\] denotes the Dieudonn\'{e} determinant of the matrix $A \otimes {\mathcal{D}_G}$ over $\mathcal{D}_G$.
\end{definition}

\medskip

\begin{theorem}[{Kielak \cite[Theorem~5.17]{Kielak2020}}]
    Let $G$ be an $L^2$-acyclic group of type $\F$ which satisfies the Atiyah conjecture. Suppose that $Y_1$ and $Y_2$ are finite $K(G,1)$-complexes. Then, 
\[\mathcal{P}(\widetilde{Y}_1; G) = \mathcal{P}(\widetilde{Y}_2; G).\]    
\end{theorem}

Hence, we may write $\mathcal{P}(G)$ to denote the $L^2$-polytope $\mathcal{P}(\tilde{Y}; G)$ for any $K(G,1)$-complex $Y$.

\medskip

We identify the elements of the infinite cyclic group $\Z$ with the integer lattice in $\RR$.

\begin{lemma}\label{lemma_polytope_abelian}
  The $L^2$-polytope $\mathcal{P}(\Z)$ of the infinite cyclic group $\Z$ is represented by $ - [0,1]$, where $[0,1] \subseteq \RR$ denotes the unit interval between 0 and 1 in $\RR$. The $L^2$-polytope of the free abelian group of rank $k$ is trivial for all $k \geq 2$.
\end{lemma}

\begin{proof}
    Let $G = \langle g \rangle$ be infinite cyclic. We identify the elements of $G$ with the integer lattice points in $\RR$ via the map $g^n \mapsto n$. The weak Whitehead group of $G$ is given by 
    \[\begin{split} \mathrm{Wh}^{w}(G) &\cong \mathbb{Q}(z, z^{-1})^{\times} / \{ \pm z^n \mid n \in \Z\} \\
    g &\mapsto z.\end{split}\]
    Then, the universal $L^2$-torsion $\rho^{(2)}_u(\RR; G) \in \mathrm{Wh}^w(G)$ of the finite free $G$-CW complex $\RR$ is identified with the element $(z-1)$ (see \cite[Example~3.3]{FriedlLueck2017}). Since $G$ is abelian, $G_{\mathrm{fab}} \cong G$ and $H_1(G; \Z) \otimes_{\Z} \RR \cong G \otimes_{\Z}\RR$. Thus, \[ \mathcal{P}(G) \simeq - \mathrm{hull}_{G \otimes_{\Z} \RR}\{0, 1\} = - [0,1].\]

    By the product formula \cite[Theorem~3.5(5)]{FriedlLueck2017}, 
    \[\rho^{(2)}_u(\RR^k; \Z^k) = \chi(\mathbb{R}^{k-1} /\, \Z^{k-1}) \cdot \rho^{(2)}_u(\mathbb{R}; 
    \Z) = 0.\]
    Hence, if $G \cong \Z^k$ for $k > 1$ then $\mathcal{P}(G)$ is trivial.
\end{proof}
\medskip

The following lemma will be our main tool for constructing polytopes.

\begin{lemma}[Friedl--L\"{u}ck {\cite[Theorem~2.5(2)]{FriedlLueck2017}}]\label{lemma_sum_polytopes}
  Let $G_1, G_2$ and $H$ be $L^2$-acyclic groups of type $\F$. Suppose that $G$ is a group which satisfies the Atiyah conjecture and which splits as the amalgamated free product \[G \cong G_1 \ast_H G_2.\] Let $i_1 \colon G_1 \to G$, $i_2 \colon G_2 \to G$ and $i_3 \colon H \to G$ be the associated embedding maps. 
Then the $L^2$-polytope of $G$ is represented by \[ {i_1}_{*}\mathcal{P}(G_1) + {i_2}_{*}\mathcal{P}(G_2) - {i_3}_{*}\mathcal{P}(H).\] \end{lemma}

\medskip

Finally, we will need the following result relating the thickness of the polytope along a character $\phi$ with the $L^2$-Euler characteristic of the kernel. Recall that we say a group $H$ is \emph{$L^2$-finite} if 
\[ \sum_{p \geq 0 }b_p^{(2)}(H) < \infty.\]


\begin{proposition}[Friedl--L\"{u}ck {\cite[Theorem~3.6(4)]{FriedlLueck2017}}, Funke {\cite[Theorem~3.52]{Funke2018}}]\label{prop:L2Euler}
    Let $G$ be an $L^2$-acyclic group of type $\F$ which satisfies the Atiyah conjecture. Then, for every epimorphism $\phi \colon G \to \Z$, the kernel $\ker \phi$ is $L^2$-finite and
    \[- \chi^{(2)}(\ker \phi) = \thi_{\phi}(\mathcal{P}(G)),\]
    where $\chi^{(2)}(\ker \phi)$ denotes the $L^2$-Euler characteristic of $\ker \phi$.
\end{proposition} 

\section{Right-angled Artin groups}

Let $L$ be a finite flag complex and let $V(L)$ denote the set of vertices and $E(L)$ the set of edges of $L$. The \emph{\raag associated to $L$} is the group $A_L$ given by the presentation  \[A_L = \langle g_v, v\in V(L) \mid [g_{v_i}, g_{v_j}] = 1 \text{ whenever }(v_i, v_j) \in E(L) \rangle.\] 
We call the set $\{g_v \mid v \in V(L)\}$ the \emph{standard generators} of $A_L$. We follow the convention that $A_{\emptyset}$ is the trivial group.

\begin{remark}
    For the remainder of this article, unless otherwise stated, every flag complex will be assumed to be finite.
\end{remark}

We record the standard fact that if $L$ has connected components $L_1, \ldots, L_k$, then $A_L$ splits as a free product 
\[A_L \cong  A_{L_1} \ast \cdots \ast A_{L_k},\]
where each $A_{L_i}$ is one-ended if and only if $L_i$ is not a singleton.

 A subcomplex $K$ of a simplicial complex $L$ is said to be \emph{induced} if for any subset of vertices $S \subseteq V(K)$ which spans a $k$-simplex in $L$, the vertices in $S$ also span a $k$-simplex in $K$. If $S \subseteq V(L)$ is a set of vertices of $L$, we write $L|_S$ to denote the induced subcomplex of $L$ spanned by the vertices $S$. It is well-known that the subgroup of the \raag $A_L$ spanned by the elements corresponding to a subset of vertices $S \subseteq V(L)$ is isomorphic to the \raag $A_{L|_S}$, 
\[\langle S \rangle_{A_L} \cong A_{L|_S}.\]
Any such subgroup is called a \emph{parabolic subgroup}.

\smallskip

If $X$ is a simplicial complex, we write $\bar{b}_i(X)$ to denote the rational rank of the reduced homology of $X$. The $L^2$-Betti numbers of right-angled Artin groups have been calculated by Davis--Leary:

\begin{theorem}[Davis--Leary \cite{DavisLeary2003}]\label{thm_l2_RAAGs}
Let $A_L$ be a \raag defined by the flag complex $L$. Then for every $p \in \Z$, the $p$-th $L^2$-Betti number of $A_L$ is 
    \[b_p^{(2)}(A_L) = \bar{b}_{p-1}(L).\]
\end{theorem}

Finally, we note that it was shown by Linnell--Okun--Schick in \cite{LinnellOkunSchick2012} that right-angled Artin group satisfy the Atiyah conjecture. 

\subsection{Characters of \raags}

Let $A_L$ be a right-angled Artin group determined by the flag complex $L$ and let $\phi \colon A_L \to \RR$ be a homomorphism. The \emph{living subcomplex} of $\phi$, denoted by $\mathcal{L}_{\phi}$, is the subcomplex of $L$ induced by the set of vertices $v$ in $L$ such that $\phi(g_v) \neq 0$. We say that a subcomplex $M$ of $L$ is \emph{dominating} if every vertex of $L$ is at distance at most one from a vertex in $M$.

\begin{proposition}[Meier--VanWyk {\cite[Theorem~6.1]{MeierVanWyk1995}}]\label{lemma_fibered_characters}
Let $A_L$ be a \raag and $\phi \colon A_L \to \RR$ a character with $\mathrm{Im}(\phi) \subseteq \mathbb{Q}$. Then $\phi$ is contained in the first Sigma invariant of $A_L$ if and only if the living subcomplex $\mathcal{L}_{\phi}$ is connected and dominating. 
\end{proposition}



\subsection{Structure of coherent right-angled Artin groups}

A simplicial complex is said to be \emph{chordal} if it does not contain an induced cycle of length greater than three. A subcomplex $K$ of a complex $L$ is \emph{separating} if there exist proper subcomplexes $L_0, L_1 \subseteq L$ which contain $K$ as a proper subcomplex, and such that $L_0 \cap L_1 = K$ and $L = L_0 \cup L_1$. For any two subcomplexes $K_0, K_1 \subseteq L$, a subcomplex $K$ of $L$ is said to \emph{separate $K_0$ from $K_1$}, or is said to be \emph{$(K_0, K_1)$-separating}, if there exist proper subcomplexes $L_0, L_1 \subseteq L$ such that $L_0 \cap L_1 = K$, $L = L_0 \cup L_1$ and $K_0 \subseteq L_0$ and $K_1 \subseteq L_1$.

We say that two disjoint subcomplexes $K_0$ and $K_1$ of a simplicial complex $L$ are \emph{adjacent} if there exists an edge connecting a vertex of $K_0$ to a vertex of $K_1$.

\begin{lemma}\label{prop_chordal}
    Let $L$ be a connected chordal flag complex. Then for any two disjoint non-adjacent subcomplexes $K_0$ and $K_1$ of $L$, there exists an induced subcomplex $K$ of $L$ which is a simplex and which separates $K_0$ from $K_1$.
\end{lemma}

\begin{proposition}\label{prop:structure_coherent_raags}
    Let $A_L$ be a \raag where the defining flag complex $L$ is connected and chordal. Then $A_L$ admits a graph-of-groups splitting $A_L \cong \pi_1(\mathbb{X})$ with the following properties: 
    \begin{enumerate}
        \item the underlying graph $X$ of $\mathbb{X}$ is a finite tree, and 
        \item each vertex and edge group in $\mathbb{X}$ corresponds to a parabolic subgroup determined by a subcomplex $L_0$ of $L$ which is a simplex. If $A_L \not\cong \Z$ then each vertex group is free abelian of rank at least two. 
    \end{enumerate}
\end{proposition}

\begin{proof}
    If $L$ is a simplex then $A_L$ is free abelian and the result is obviously true. 

    Suppose that $L$ is not a simplex. Then there exist two non-adjacent vertices in $L$. Hence, by \cref{prop_chordal}, there exists a separating simplex $K \subseteq L$. Let $L_0$ and $L_1$ be the induced subcomplexes of $L$ such that $L_0 \cap L_1 = K$ and $L = L_0 \cup L_1$. Then $A_L$ splits as an amalgamated free product $A_{L} \simeq A_{L_0} \ast_{A_{K}}  A_{L_1}$. Since $K$ is an induced subcomplex of $L_0$ and $L_1$, the free basis of $A_{K}$ embeds in the standard bases of the right angled Artin groups $A_{L_0}$ and $A_{L_1}$. 

    We continue refining the splitting in this way until all the vertex groups are free abelian. 
\end{proof}



\begin{corollary}\label{prop_subgroup_Betti_numbers}
    Let $L$ be a chordal flag complex and $A_L$ the corresponding \raagstop Let $H \leq A_L$ be a non-trivial finitely generated subgroup. Then $H$ is of type $\F$ and the $L^2$-Betti numbers of $H$ satisfy $b_p^{(2)}(H) = 0$ for every $p \neq 1$.
\end{corollary}

\begin{proof}
    Let us first assume that $L$ is connected. Let $H \leq A_L$ be a non-trivial finitely generated subgroup. Let $T$ be the Bass--Serre tree corresponding to the splitting of $A_L$ obtained in \cref{prop:structure_coherent_raags}. If $H$ is elliptic in $T$, then $H$ is a subgroup of a finitely generated free abelian group and thus it is itself finitely generated and free abelian. Hence $H$ is of type $\F$ and $b_p^{(2)}(H) = 0$ for all $p >0$. Moreover, since $H$ is non-trivial, it must be infinite and thus $b_0^{(2)}(H) = 0$.
    
    Suppose now that the action of $H$ on $T$ contains a hyperbolic element. Since $H$ is finitely generated, there exists a minimal subtree $T_0 \subseteq T$ which is preserved by $H$. The stabiliser of every edge and vertex in $T_0$ under the action of $H$ is also finitely generated and free abelian. Since $H$ is finitely generated, the action of $H$ on $T_0$ must be cocompact. It follows that $H$ splits as a finite graph of groups with edge and vertex groups finite rank free abelian and thus of type $\F$, and hence $H$ is of type $\F$. Applying the Mayer--Vietoris sequence to the splitting, we obtain that $b_p^{(2)}(H) = 0$ for all $p \geq 2$. Moreover, $H$ is infinite and thus $b_0^{(2)}(H) = 0$.

    Suppose that $L$ is disconnected and let $L_1, \ldots, L_k$ be the connected components of $L$. Let $T'$ be the Bass--Serre tree corresponding to the free splitting $A_L = A_{L_1} \ast \ldots \ast A_{L_k}$. If the subgroup $H \leq A_{L}$ is elliptic in the tree $T'$, then $A_L$ is a finitely generated subgroup of a \raag with a connected chordal defining complex. Thus, it is of type $\F$ and its $L^2$-homology is concentrated in the first degree by the previous paragraphs.

    Suppose now that the action of $H$ on $T'$ contains a hyperbolic element. Let $T_0' \subseteq T_0$ be the minimal subtree preserved by $H$. Then $H$ acts minimally on $T_0'$ with trivial edge stabilisers. Since $H$ is finitely generated, it must be the case that the action is cocompact and each vertex stabiliser is a finitely generated subgroup of a conjugate of $A_{L_i}$ for some $i$. Hence, each vertex stabiliser is of type $\F$ and is either trivial or its $L^2$-homology is concentrated in the first degree, by the previous arguments. It follows that $H$ is of type $\F$ and by the Mayer--Vietoris sequence it has $L^2$-homology concentrated in the first degree. 
\end{proof}

\begin{corollary}\label{cor_coherence_characterisation}
    Let $L$ be a flag complex and $A_L$ the corresponding \raagstop Then, the following are equivalent:
    \begin{enumerate}
        \item \label{4.7:it1} the flag complex $L$ is chordal;
        \item \label{4.7:it2} every finitely generated subgroup of the \raag $A_L$ is of type $\F$;
        \item \label{4.7:it3} every finitely generated subgroup of the \raag $A_L$ is finitely presented.
    \end{enumerate}
\end{corollary} 

\begin{proof}
    The implication \eqref{4.7:it1} $\Rightarrow$ \eqref{4.7:it2} is proved in \cref{prop_subgroup_Betti_numbers}, and \eqref{4.7:it2} $\Rightarrow$ \eqref{4.7:it3} is obvious. Finally, that \eqref{4.7:it3} $\Rightarrow$ \eqref{4.7:it1} is a classical result due to Droms \cite{Droms1987}.
\end{proof}

We finish this section by applying a similar argument to the proof of \cref{prop_subgroup_Betti_numbers} to study $L^2$-Betti numbers of kernels of homomorphisms $\phi \colon A_L \to \Z$, without any assumptions on the finiteness properties of $\ker \phi$.

\begin{lemma}\label{lemma_vanishing_l2_betti_numbers_kernel}
    Let $G$ be a coherent one-ended \raag and let 
    $\phi \colon G \to \Z$ be a non-trivial map. Then for every $p \neq 1$, $b_p^{(2)}(\ker \phi) = 0$.
\end{lemma}

\begin{proof}

    As in the proof of \cref{prop_subgroup_Betti_numbers}, let $T$ be the Bass--Serre tree corresponding to the splitting of $A_L$ obtained in \cref{prop:structure_coherent_raags} and consider the induced action of $\ker \phi$ on $T$. Since $\ker \phi$ is a normal subgroup of $G$, the induced action is minimal. Then, the action of $\ker \phi$ on $T$ induces a splitting of $\ker\phi$ as a graph of finite rank free abelian groups \[\ker \phi \cong \pi_1(\mathbb{X}).\]
    Pick any ascending chain of finite sub-graphs of groups of $\mathbb{X}$, 
    \[\mathbb{X}^{(0)} \subseteq \mathbb{X}^{(1)} \subseteq \ldots \subseteq \mathbb{X}\]
    such that $\mathbb{X} = \bigcup_{i} \mathbb{X}^{(i)}$.
    By the Mayer--Vietoris sequence, each group $\pi_1(\mathbb{X}^{(i)})$ has $L^2$-homology concentrated in the first dimension. Then, 
    \[H_p(\ker \phi ; \mathcal{D}_G) = \varinjlim H_p(\pi_1(\mathbb{X}^{(i)}); \mathcal{D}_G)\]
    where $\mathcal{D}_G$ is the Linnell ring of $G$. Hence for any $p \neq 1$, $H_p(\ker \phi ; \mathcal{D}_G) = 0$.

    Finally, since $\ker \phi$ is a subgroup of a right-angled Artin group, it is torsion free and satisfies the Atiyah conjecture. Hence, denoting by $\mathcal{D}_K$ the Linell ring of $\ker \phi$, we have that
    \[b_p^{(2)}(\ker \phi) = \mathrm{rank}_{\mathcal{D}_K} H_p(\ker \phi; \mathcal{D}_K) = \mathrm{rank}_{\mathcal{D}_G} H_p(\ker \phi; \mathcal{D}_G),\]
    where the first equality follows from \cref{lueck} and the second follows from the fact that $\mathcal{D}_G$ is a $\mathcal{D}_K$-module and thus, in particular, is flat over $\mathcal{D}_K$. Thus,  $b_p^{(2)}(\ker \phi) = 0$ for any $p \neq 1$.
    
\end{proof}

\subsection{$L^2$-polytopes of coherent one-ended \raags}

Let $L$ be a flag complex on at least two vertices and let $v \in V(L)$ be a vertex. We write $L \setminus \{v\}$ to denote the flag complex obtained from $L$ by removing every simplex of $L$ which contains $v$ as a face. We define the \emph{cut rank} of a vertex $v$ to be 
\[\cut_L(v) = \bar{b}_0 (L \setminus \{v\}) \in \NN.\]
Recall that for any complex $Y$, $\bar{b}_0(Y)$ is one less than the number of connected components of $Y$.

 A vertex $v$ of $L$ is said to be a \emph{cut vertex} if $\cut_L(v) > 0$.

\begin{proposition} \label{lemma:L2_polytope_RAAG}
    Let $A_L$ be a one-ended coherent \raag determined by the flag complex $L$ on at least two vertices. Then $A_L$ is $L^2$-acyclic and its $L^2$-polytope $\mathcal{P}(A_L) \in \mathfrak{P}_T(A_L)$ is represented by the polytope 
\[ \sum_{v \in V(L)} \cut_L(v) \cdot [o,\pi(g_v)],\]
where $\pi$ is the natural map
\[A_L \to   H_1(A_L ; \RR),\]
and $[o, \pi(g_v)]$ is the line joining the origin to $\pi(g_v)$ in $H_1(A_L ; \RR)$. In particular, $\mathcal{P}(A_L)$ is a single polytope. 
\end{proposition}

\begin{proof}[Proof of \cref{lemma:L2_polytope_RAAG}]
    Since $A_L$ is one-ended and coherent, the flag complex $L$ is connected and chordal. It follows that $L$ is contractible. Thus, by \cref{thm_l2_RAAGs} all $L^2$-Betti numbers of $A_L$ vanish.

    If $L$ has exactly two vertices then each vertex of $L$ has cut rank equal to 0. Moreover, the group $A_L$ is free abelian of rank two and thus by \cref{lemma_polytope_abelian} the $L^2$-polytope of $A_L$ is trivial. Hence the Lemma holds when $|V(L)| = 2$.
    
Suppose that $L$ has at least three vertices. We argue by induction on the number of cut vertices in $L$. Consider the graph-of-groups splitting $A_L \cong \pi_1(\mathbb{X})$ obtained in  \cref{prop:structure_coherent_raags}. If $L$ has no cut vertices then by \cite[Theorem~A]{Clay2014}, $A_L$ does not split over $\Z$. Hence, each edge group in the splitting is free abelian of rank at least two. Since every right-angled Artin group satisfies the Atiyah conjecture, we can repeatedly apply \cref{lemma_sum_polytopes} to conclude that $\mathcal{P}(A_L)$ is of the form \[
   \sum_{v \in V(X)} \mathcal{P}(\mathbb{X}_v) - \sum_{e \in E(X)} \mathcal{P}(\mathbb{X}_e)
 \]
where each $\mathbb{X}_e$ and  $\mathbb{X}_v$ is free abelian and of rank at least two. Thus, by \cref{lemma_polytope_abelian}, the $L^2$-polytope of $A_L$ is represented by a difference of singletons and thus is trivial in $\mathfrak{P}_T(A_L)$.

    Let $n$ be a positive integer and suppose that the result holds for all one-ended coherent $A_L$ where $L$ has less than $n$ cut vertices. Suppose that $L$ has $n$ cut vertices and let $v \in V(L)$ be a cut vertex. Then, there exist induced subcomplexes $L_0, \ldots, L_{\cut_L(v)}$ of $L$, such that 
    \[ L = \bigcup_{0 \leq i \leq \cut_L(v)} L_{i},\]
    and $L_i \cap L_j = \{v\}$ for all $i \neq j$. 
    
    This induces a graph-of-groups splitting of $A_L$, where the underlying graph $\Gamma$ is a star with the central vertex and edges labelled by the infinite cyclic group $\langle g_v \rangle$, and the remaining vertices labelled by the \raags $A_{L_j}$ for all $0 \leq j \leq \cut_L(v)$. 
    Applying \cref{lemma_sum_polytopes} inductively to this splitting, we obtain 
    \[ \mathcal{P}(A_L) \simeq \left(\sum_{0 \leq j \leq \cut_L(v)} (\iota_j)_{*} \mathcal{P}(A_{L_j}) \right) - \cut_L(v) \cdot \iota_{*}\mathcal{P}(\langle g_v \rangle),\]
    where $\iota_{*} \colon \mathfrak{P}_T(\langle g_v \rangle) \to \mathfrak{P}_T(A_L)$ and $(\iota_j)_{*} \colon \mathfrak{P}_T(A_{L_j}) \to \mathfrak{P}_T(A_L)$, for $0 \leq j \leq \cut_L(v)$, denote the maps on polytope groups induced by subgroup inclusions.

    By \cref{lemma_polytope_abelian}, 
    \[ i_{*}\mathcal{P}(\langle g_v \rangle) \simeq - [o, \bar{g}_v],\]
    where $\bar{g}_v$ denotes the image of the element $g_v \in A_L$ in $H_1(A_L; \Z) \otimes_{\Z} \RR$.

    Note that for every subcomplex $L_j$, we have that $\cut_{L_j}(v) = 0$, and for every vertex $w$ of $L_j$ not equal to $v$, there is a bijection between the connected components of $L_j \setminus \{w\}$ and $L \setminus \{w\}$. Hence, we have that
    \[ \cut_{L_j}(w) = \cut_{L}(w).\]
    In particular, each $L_j$ has strictly less than $n$ cut vertices. Since each $L_j$ is a chordal and connected complex on at least two vertices, by induction it follows that
    \begin{equation}\begin{split}\mathcal{P}(A_{L_j}) &\simeq \sum_{w \in V(L_j)} \cut_{L_j}(w) \cdot [o, 
    \pi_j(g_w)] \\
    &= \sum_{w \in V(L_j) \setminus \{v\} } \cut_{L_j}(w) \cdot [o, 
    \pi_j(g_w)]\end{split}\end{equation}
    where $\pi_j \colon A_{L_j} \to H_1(A_{L_j}; \Z) \otimes_{\Z} \RR$ is the composition of the abelianisation map $A_{L_j} \twoheadrightarrow H_1(A_{L_j}; \Z)$ and the map $ H_1(A_{L_j}; \Z) \to H_1(A_{L_j}; \Z) \otimes_{\Z} \RR$. Moreover, for each $g_w$ we have that 
    \[(i_j)_{*}([o, \pi_j(g_w)]) = [o, \pi(g_w)]. \]
By combining the arguments above, it follows that 
    \[\mathcal{P}(A_L) = \sum_{w \in V(A_L) \setminus \{v\}} \cut_{L}(w) \cdot [o, \pi(g_w)] + \cut_{L}(v) \cdot [o, \pi(g_v)].\qedhere \]\end{proof}

    \begin{remark}
    After writing this article, we learned that Okun--Schreve computed the $L^2$-polytope for all $L^2$-acyclic right-angled Artin groups in \cite[Corollary~6.5]{OkunSchreve2024}. Since the methods employed in this article are different to those in \cite{OkunSchreve2024}, we decided to keep the proof of \cref{lemma:L2_polytope_RAAG}.
\end{remark}

    \medskip

\begin{remark} We remark that the $L^2$-polytope of coherent \raags in general does not detect algebraic fibering of characters $A_L \to \Z$ (e.g. in the sense of \cite[Theorem~5.29]{Kielak2020}). 

Indeed, let $L$ be the flag complex obtained by gluing two 2-simplices along an edge. The corresponding \raag is given by \[A_L \cong \Z^3 \ast_{\Z^2} \Z^3.\] The group $A_L$ is one-ended and coherent. The cut rank of each vertex in $L$ is zero and thus by \cref{lemma:L2_polytope_RAAG} the $L^2$-polytope of $A_L$ is trivial.

    Let $v_1$ and $v_2$ denote the endpoints of the common edge of the two 2-simplices in $L$, and let $w_1$ and $w_2$ be the remaining two vertices. Then, the set $\{g_{v_1}, g_{v_2}, g_{w_1}, g_{w_2}\}$ descends to a free generating set of the abelianisation of $A_L$, and 
    \[ H^1(A_L; \RR) \simeq (g_{v_1})^{*} \RR \oplus (g_{v_2})^{*} \RR \oplus (g_{w_1})^{*} \RR \oplus (g_{w_2})^{*} \RR.\]
 By \cref{lemma_fibered_characters}, the first Sigma invariant of $A_L$ is
\[\Sigma^1(A_L)  = \{ \phi \in \Hom(A_L, \RR) - 0 \mid \phi(g_{v_1}) \neq 0 \text{ or }\phi(g_{v_2})\neq 0 \}.\]
In particular, a non-empty proper subset of characters of $A_L$ is algebraically fibered. Meanwhile, $\mathcal{P}(A_L)$ is trivial and thus it does not encode any information that distinguishes the characters.
\end{remark}

\medskip

\begin{proposition}\label{prop:l2_betti_numbers_general_kernels}
     Let $G$ be a coherent one-ended \raag and $\phi \colon G \to \Z$ an epimorphism. Then, for every non-negative integer $p \geq 0$, \[ b_p^{(2)}(\ker \phi) = \left\{\begin{array}{lr}
        0, & \text{if }  p \neq 1\\
        \sum_{v \in V(L)} |\phi(v)|\bar{b}_{0}(\mathrm{Lk}(v)), & \text{if } p = 1.\\
        \end{array}\right. \]
\end{proposition}

\begin{proof}

Note that since $G$ is assumed to be one-ended, $L$ has at least two vertices. The vanishing of $b_p^{(2)}(\ker \phi)$ is proved in \cref{lemma_vanishing_l2_betti_numbers_kernel}. In particular, we have that
    \[ \chi^{(2)}(\ker \phi) = - b_1^{(2)}(\ker \phi).\]
Hence by \cref{prop:L2Euler}, 
    \[ \mathrm{th}_{\phi}(\mathcal{P}(G)) = b_1^{(2)}(\ker \phi),\]
    where $\mathcal{P}(G)$ denotes the $L^2$-polytope of $G$.
    Moreover, by \cref{lemma:L2_polytope_RAAG}, the $L^2$-polytope $\mathcal{P}(G)$ is represented by the polytope 
    \[ \sum_{v \in V(L)} \cut_L(v) \cdot [o,\pi(g_v)] \\
= \sum_{v \in V(L)} \bar{b}_0(\mathrm{Lk}(v)) \cdot [o,\pi(g_v)],\]
where $\pi(g_v)$ is the image of the generator $g_v$ of $A_L$ corresponding to $v$ under the natural map $A_L \to   H_1(A_L ; \RR)$ and $[o,\pi(g_v)]$ is the line segment joining the origin to $\pi(g_v)$. It follows that for any epimorphism $\phi \colon G \to \Z$,
\[\mathrm{th}_{\phi}(\mathcal{P}(G)) = \sum_{v \in V(L)} \bar{b}_0(\mathrm{Lk}(v)) \cdot |\phi(g_v)|.\]
The result now follows.
\end{proof}

Fisher--Hughes--Leary show that for any right-angled Artin group $A_L$ and any homomorphism $\phi \colon A_L \to \Z$ with kernel of type $\FP_n(\mathbb{Q})$, the $n$-th $L^2$-Betti number of the kernel is given by
      \begin{equation}\label{eq:l2} b_n^{(2)}(\ker \phi) = \sum_{v \in V(L)} |\phi(v)|\bar{b}_{n-1}(\mathrm{Lk}(v)).\end{equation}
In \cref{prop:l2_betti_numbers_general_kernels}, we show that \eqref{eq:l2} holds for the kernel of any epimorphism $\phi \colon A_L \to \Z$ where $A_L$ is one-ended and coherent, without any restrictions on the finiteness properties of the kernel.

Note also that by \cref{prop:L2Euler}, for any $L^2$-acyclic \raag and any epimorphism $\phi \colon A_L \to \Z$, the $L^2$-Betti numbers of $\ker \phi$ are finite. This leads us to ask the following:

\begin{question}
    Let $A_L$ be an $L^2$-acyclic \raag and $\phi \colon A_L \to \Z$ an epimorphism. Is it always true that 
    \[b_n^{(2)}(\ker \phi) = \sum_{v \in V(L)} |\phi(v)|\bar{b}_{n-1}(\mathrm{Lk}(v))?\]
\end{question}

\medskip

By \cref{thm_l2_RAAGs}, the $L^2$-Euler characteristic of a \raag defined by the flag complex $L$ is 
\[\chi(A_L) = 1 - \chi(L) = - \widebar{\chi}(L).\]  Combining this observation with \cref{prop:l2_betti_numbers_general_kernels} we obtain the following:

\begin{corollary}\label{cor:l2_Euler_norm}
    Let $G$ be a coherent one-ended \raag and $\phi \colon G \to \Z$ an epimorphism. Then,
    \[\chi^{(2)}(\ker \phi) = \sum_{v \in V(L)} |\phi(g_v)| \chi(A_{\mathrm{Lk}(v)}).\]
\end{corollary}

\section{Splitting complexity and $L^2$-Euler characteristic}

\subsection{Efficient splittings}

A graph of groups $\mathbb{X}$ is said to be of \emph{finite type} if the underlying graph is finite and every edge and vertex group is of type $\FP$. 
 


Let $\phi \colon G \to \Z$ be a homomorphism. A graph-of-groups splitting $G \cong \pi_1(\mathbb{X})$ is said to be \emph{dual} to $\phi$ if every vertex group is identified with a subgroup of $\ker \phi$.

\begin{definition}\label{def:complexity_of_Gtree} Let $G$ be a group and let $\phi \colon G \to \Z$ be a homomorphism. 
   Let $\mathbb{X}$ be a graph of groups of finite type with a single vertex. The \emph{complexity} of a graph-of-groups splitting $G \cong \pi_1(\mathbb{X})$ which is dual to $\phi$ is
\[c(\mathbb{X}; \phi) = - \sum_{e \in E^{+}(X)} |\phi(t_e)|\chi(\mathbb{X}_e),\]
where $t_e \in G$ is the group element corresponding to the edge $e \in E^{+}(X)$.
\end{definition} 
Note that if the graph of groups $\mathbb{X}$ has no edges then $c(\mathbb{X}; \phi) = 0$.

\begin{remark}
    There is a natural extension of \cref{def:complexity_of_Gtree} to graph-of-groups splittings $G \cong \pi_1(\mathbb{X}, \mathcal{T})$ where the underlying graph can have more than one vertex. In the setting where all possible edge groups have non-positive Euler characteristic, collapsing the non-loop edges in a graph of groups with more than one vertex results in a splitting with the same or lower complexity. Thus, we only consider single-vertex splittings. 
\end{remark}

The aim of this subsection is to prove the following:

\begin{proposition}\label{thm_realising_splittings}
    Let $L$ be a chordal flag complex. Let $G = A_L$ be the \raag defined by $L$ and let $\phi \colon G \to \Z$ be a non-trivial homomorphism. Then, there exists a single-vertex graph-of-groups splitting $G \cong \pi_1(\mathbb{X})$ of finite type which is dual to $\phi$, and such that
    \[ c(\mathbb{X}; \phi) = - \sum_{v \in V(L)} |\phi(g_v)|\chi(A_{\mathrm{Lk}(v)}).\]
\end{proposition}




Recall that if $L$ is a simplicial complex and $K$ is a subcomplex of $L$ then we write $L \setminus K$ to denote the simplicial complex obtained by removing every simplex in $L$ which contains a simplex of $K$ as a face.

For any simplicial complex $L$ and for any two vertices $x,y \in V(L)$, we write 
\[d_{L}(x,y) \in \Z \cup \{\infty\}\]
to denote the minimal number of edges in a path between $x$ and $y$ in the 1-skeleton $L^{(1)}$ of $L$. If no such path exists then we set $d_L(x,y) = \infty$. For any induced subcomplex $M$ of $L$, the \emph{one-neighbourhood of $M$ in $L$}, denoted by $N_1(M)$, is the induced subcomplex of $L$ spanned by the set of vertices of distance at most one from some vertex in $M$.

\begin{proof}[Proof of \cref{thm_realising_splittings}]
    We begin by assuming that the flag complex $L$ is connected.

    If $L$ consists of a single vertex $v$, then $A_L = \langle g_v \rangle \cong \Z$ and $\phi(g_v) = k \in \mathbb{Z} \setminus 0$. Hence, 
    \[ \sum_{v \in V(L)} |\phi(g_v)| \chi(A_{\mathrm{Lk}(v)}) = |k| \cdot \chi(A_{\emptyset}) = |k|.\]
    Now, set $\mathbb{X}$ to be a graph of groups where the underlying graph has a single vertex and single edge $e$, both labelled by the trivial group. Then, the splitting $A_L \cong \pi_1(\mathbb{X})$, where the edge loop $t_e$ is identified with the generator $g_v$ of $A_L$, is dual to $\phi$, and
    \[ c(\mathbb{X}; \phi) = - |k|.\] Hence, the result holds in this case. 

    Suppose now that $L$ has at least two vertices. A \emph{living block} in $L$ is defined to be a connected component of the living subcomplex $\mathcal{L}_{\phi}$, i.e. a maximal connected induced subcomplex $B$ of $L$ such that $\phi(g_v) \neq 0$ for every vertex $v$ in $L$. Let $\{B_i\}_{i \in I}$ be the set of living blocks in $L$, and let $\{C_j\}_{j\in J}$ be the connected components of the complement $L \setminus \bigcup_{i \in I} B_i$. For each living block $B$, let $N_1(B)$ be the one-neighbourhood of $B$ in $L$.

    \begin{lemma}
        If $C \in \{C_j\}_{j \in J}$ and $N_1(B) \cap C \neq \emptyset$ for some living block $B$, then $N_1(B)$ and $C$ intersect in a connected subcomplex of $L$.
    \end{lemma}

    \begin{proof}
        
    Suppose this is not the case.  Let $x , y \in V(N_1(B)) \cap V(C)$ be contained in two distinct connected components and suppose that they minimise the quantity 
    \[d_{N_1(B)}(x,y) + d_C(x,y)\]
    over all such points $x,y$.
    
    Since $N_1(B)$ and $C$ are induced subgraphs, there is no edge joining $x$ to $y$ in $L$, since then it would be contained in $N_1(B) \cap C$, and thus $x$ and $y$ would be in the same connected component of $N_1(B) \cap C$. Let $\alpha$ be a shortest path in $N_1(B)$ joining $x$ to $y$, and similarly let $\beta$ be a shortest path in $C$. Let $\gamma$ be the loop at $x$ obtained by concatenating $\alpha$ followed by the inverse of $\beta$, which we denote by $\bar{\beta}$. The image of the loop $\gamma$ has at least four edges. Since $\alpha$ is the shortest path joining $x$ to $y$ in $N_1(B)$, there is no edge in $N_1(B) \setminus \mathrm{im}(\alpha)$ between two vertices of $\alpha$, and similarly for $\beta$. Then, since $L$ is chordal, there must exist an edge joining a vertex $u$ in the interior of $\alpha$ to a vertex $v$ in the interior of $\beta$. 
    
    Suppose first that $\phi(u) = 0$. Then, maximality of $C$ implies that $u \in C$. Hence, $u \in C \cap N_1(B)$, and we have that $d_{N_1(B)}(u,y) < d_{N_1(B)}(x,y)$, and \[\begin{split}d_C(u,y) &\leq d_C(u,v) + d_C(v,y)\\  &\leq d_C(x,y).\end{split}\]
    This contradicts the minimality of $d_{N_1(B)}(x,y) + d_C(x,y)$. Hence, it must be the case that $\phi(u) \neq 0$ and thus $u \in B$. However, since $d_L(u,v) = 1$ we have that $v \in N_1(B)$. Hence $v \in N_1(B) \cap C$. By the same arguments as above, we arrive at a contradiction in this case also.


    \end{proof}

    A similar argument shows that if $B$ and $B'$ are distinct living blocks then $N_1(B) \cap N_1(B')$ is connected. 

    \medskip

    Let $(X, V_0(X), V_1(X))$ be a bipartite graph where the vertices in $V_0(X)$ correspond to the subcomplexes $N_1(B_i)$ for all living blocks $B_i$, and the vertices in $V_1(X)$ correspond to the subcomplexes $C_j$. The edges between the vertices corresponding to $N_1(B_i)$ and $C_j$ are in one-to-one correspondence with the connected components of their intersection, $N_1(B_i) \cap C_j$. Note that the graph $X$ is simplicial since each intersection $N_1(B_i) \cap C_j$ is connected. Moreover, a loop in $X$ corresponds to an induced cycle in $L$. Since every cycle in a bipartite simplicial graph has length at least 4, it follows that $X$ must be a tree.

    We now define a graph of groups $\mathbb{X}$ with the underlying graph given by $(X, V_0(X), V_1(X))$, as follows. For each vertex $v$ of $X$, the vertex group $\mathbb{X}_v$ is the \raag determined by the corresponding flag complex, and similarly for each edge group. The edge inclusions are determined by the inclusions of the corresponding flag complexes. Then, there is a monomorphism 
    \[\pi_1(\mathbb{X}) \to A_L\]
    defined by subgroup inclusion on each vertex group. Since $L$ is covered by the subcomplexes $N_1(B_i)$ and $C_j$, the homomorphism is onto. 
    
    By construction, for each vertex $v \in V_1(X)$, the corresponding vertex group $\mathbb{X}_v$ is identified with a subgroup of $\ker \phi$. 
    
    Let $w \in V_0(X)$ be a vertex which corresponds to the subcomplex $N_1(B)$ in $L$, for some living block $B$. Let $\phi_B$ denote the restriction of the homomorphism $\phi$ to the subgroup $A_{N_1(B)}$, which we consider as an epimorphism onto its image. 

    Then, $\phi_B$ induces an HNN splitting of the \raag $A_{N_1(B)}$, with base the subgroup $\ker \phi_B = \ker \phi \cap A_{N_1(B)}$, and such that both of the edge inclusions are onto. Note that all the edge groups in the graph of groups $\mathbb{X}$ are contained in $\ker \phi$.  Hence, the incident edge groups at the vertex corresponding to $N_1(B)$ are contained in $\ker \phi \cap A_{N_1(B)}$. Thus, we may replace a vertex in $X$ labelled by the group $A_{N_1(B)}$ by an edge loop corresponding to the HNN extension $A_{N_1(B)} = \ker \phi_B \ast_{\ker \phi_B}$. We repeat this procedure at every vertex, eventually arriving at a graph of groups $\mathbb{X}'$ where every vertex group is contained in $\ker \phi$.

    Finally, we collapse every non-loop edge in the graph of groups $\mathbb{X}'$ to obtain a graph of groups $\mathbb{X}''$ where the underlying graph $X''$ has a single vertex. The edge groups in $\mathbb{X}''$ are identified with the subgroups of $A_L$ of the form $\ker \phi_B$, as $B$ ranges over the living blocks. For any living block $B$, the kernel of $\phi_B$ is finitely generated by \cref{lemma_fibered_characters}. Since $N_1(B)$ is an induced subcomplex of a chordal complex, it is itself chordal and thus by \cref{cor_coherence_characterisation}, $\ker \phi_B$ is of type $\F$. Hence, the Bass--Serre tree $T$ corresponding to the graph of groups $\mathbb{X}''$ is of finite type and is dual to $\phi$.

    For each living block $B$, let $e_B$ denote the corresponding edge in $X''$ (with either orientation). Let $g_B \in \pi_1(\mathbb{X}'')$ denote the group element corresponding to the edge $e_B$ and let $k_B = |\phi(g_B)|$. Then, \[k_B\cdot \Z = \mathrm{Im}(\phi(A_{N_1(B)})) \leq \Z,\] and $k_B^{-1}\cdot \phi_B$ is an epimorphism. Hence, by \cref{cor:l2_Euler_norm} and \cref{lemma:type_F_euler}, we have that
    \[ \openup 2\jot
    \begin{split} \chi(  \ker \phi_B) &= \chi(\ker k_B^{-1} \cdot \phi_B) \\ &= k_B^{-1}\sum_{v \in V(N_1(B))} |\phi(g_v)| \chi(A_{\mathrm{Lk(v)}}) \\
    &= k_B^{-1}\sum_{v \in V(B)} |\phi(g_v)| \chi(A_{\mathrm{Lk(v)}}).\end{split}\]

       On the other hand, the complexity of the graph-of-groups splitting $A_L \cong \pi_1(\mathbb{X}'') $ is given by \[\begin{split}c(\mathbb{X}''; \phi) &= -\sum_{e \in E^{+}(X)} |\phi(t_e)| \cdot \chi(\mathbb{X}''_e) \\
        &= \sum_{B \text{ living block}}|\phi(t_e)| \cdot \chi(\ker \phi_B) \\
        &= \sum_{B \text{ living block}} \sum_{v \in V(B)} |\phi(g_v)| \chi(A_{\mathrm{Lk}(v)})  \\
        &= \sum_{v \in V(L)} |\phi(g_v)| \chi(A_{\mathrm{Lk}(v)}). \end{split}
        \] This completes the proof when $L$ is connected.

        \medskip

    Suppose now that $L$ has connected components $L_1, \dots, L_k$. For each component $L_i$, if $\phi(A_{L_i}) = 0$ then let $\mathbb{X}^{(i)}$ be a graph-of-groups with a single vertex labelled by $A_{L_i}$. Otherwise,
let $\mathbb{X}^{(i)}$ be the graph-of-groups splitting of $A_{L_i}$ of finite type and that is is dual to $\phi|_{A_{L_i}} \colon A_{L_i} \to \Z$, as constructed in the previous part of the proof. 
    
    Construct a graph of groups $\mathbb{X}$ by taking the disjoint union of the $\mathbb{X}^{(i)}$ and a single vertex $v_0$ labelled by the trivial group. Add a directed edge $e_i$ from $v_0$ to the unique vertex of $\mathbb{X}^{(i)}$ for every $1 \leq i \leq k$. Label each edge $e_i$ with the trivial group. Then,
    \[ \pi_1(\mathbb{X}) \cong A_{L_1} \ast \cdots \ast A_{L_k},\]
    and the splitting is dual to $\phi$.
    Moreover, the graph of groups $\mathbb{X}$ is of finite type. We collapse the non-loop edges to obtain a graph of groups of finite type with a single vertex, which we continue denoting by $\mathbb{X}$.

    Then, the splitting complexity of $\mathbb{X}$ is  

    \[ \begin{split}c(\mathbb{X}; \phi) &= - \sum_{1 \leq i \leq k} \sum_{e \in E^{+}(X^{(i)})} |\phi(t_e)| \chi(\mathbb{X}_e) \\
    &= \sum_{1 \leq i \leq k} c(\mathbb{X}^{(i)}; \phi) \\
    &= \sum_{1 \leq i \leq k} \sum_{v \in V(L_i)} | \phi(g_v)| \chi(A_{\mathrm{Lk}_{L_i}(v)}) \\
    &= \sum_{v \in V(L)} | \phi(g_v)| \chi(A_{\mathrm{Lk}_{L}(v)}). \qedhere
    \end{split}\] \end{proof}
\medskip

\subsection{$L^2$-Betti number bounds}

\begin{proposition}\label{prop_L2Betti}
    Let $G$ be a torsion-free group which satisfies the Atiyah conjecture and let $\phi \colon G \to \Z$ be an epimorphism. Let $G \cong \pi_1(\mathbb{X})$ be a graph-of-groups splitting where the underlying graph $X$ has a single vertex $v$ and for each edge $e \in E(X)$, $t_e \in G$ is the corresponding group element. Fix a positive integer $n$.
    
    Suppose that the vertex group $\mathbb{X}_v$ is identified with a finitely generated subgroup of $\ker \phi$ and $\phi(t_e) \neq 0$ for all $e \in E(X)$. Moreover, assume that $b_{n-1}^{(2)}(\mathbb{X}_e) = 0$ for all $e \in E(X)$ and $\sum_{e \in E^{+}(X)} \lb_n(\mathbb{X}_e) = \lb_n(\mathbb{X}_v)$. Finally, suppose that $b_{n+1}^{(2)}(H) = 0$ for every finitely generated subgroup $H \leq G$.
    
    Then, 
\[ \lb_n(\ker \phi) \leq \sum_{e \in E^{+}(X)} |\phi(t_e)| \cdot \lb_n(\mathbb{X}_e). \qedhere\]
\end{proposition}
\begin{proof}

   Let $T$ be the Bass--Serre tree corresponding to the graph-of-groups splitting $G \cong \pi_1(\mathbb{X}, v)$. Consider the induced action of the subgroup $\ker \phi \leq G$ on the tree $T$ and let $Y:= \ker \phi \backslash T$ be the quotient. Let $\mathbb{Y}$ be the resulting graph of groups. 
  Fix a spanning tree $\mathcal{T} \subseteq Y$ and an identification $\ker \phi \simeq \pi_1(\mathbb{Y},\mathcal{T}).$ Since $\ker \phi$ is a normal subgroup of $G$, there is an induced action of the infinite cyclic quotient $H := G / \ker \phi$ on the graph $Y$ and the quotient graph $H \backslash Y$ is isomorphic to $X := G \backslash T$. 
  
  Let $\bar{t}$ denote a generator of $H$ and let $t \in G$ be a lift of $\bar{t}$ to $G$. Since the stabiliser of every vertex of $T$ under the action of $G$ is contained in $\ker \phi$, no power of $t$ fixes a vertex of $T$. Moreover, since $H \backslash Y \cong X$, there is a single orbit of vertices for the action of $H$ on $Y$.
 
 Let $v_0$ be a lift of $v$ to $Y$. Let $v_i := t^{i} \cdot v_0$ for each $i \in \Z$. Define $Y^{(k)}$ to be the subgraph of $Y$ spanned by the set of vertices $\{v_{i} \mid -k \leq i \leq k\}$. For every edge $e \in E(X)$, let $t_e \in G$ be the corresponding group element. The set of lifts of $e$ to $Y$ is the collection of edges connecting the vertex $v_i$ to $v_{i + \phi(t_e)}$ for each $i \in \Z$. Since $\mathrm{gcd}_{e\in E^{+}(X)}\{\phi(t_e) \} = 1$, there exists an integer $M > 0$ such that $Y^{(k)}$ is connected for all $k \geq M$. We will assume that $M > \mathrm{max}_{e\in E^{+}(X)}\{|\phi(t_e)|\}$.

 Hence, the graphs $Y^{(k)}$ for $k \geq M$ form a nested sequence
    \[Y^{(M)} \subseteq Y^{(M+1)} \subseteq Y^{(M+2)} \subseteq \ldots \]
of compact connected graphs, such that $\bigcup_{k \geq M} Y^{(k)} = Y$.
Let $\mathbb{Y}^{(k)}$ be the sub-graph-of-groups determined by the inclusion $Y^{(k)} \subseteq Y$. The subgraph $\mathcal{T}_k := \mathcal{T} \cap Y^{(k)}$ is a spanning tree of $Y^{(k)}$ and the fundamental group $\pi_1(\mathbb{Y}^{(k)}, \mathcal{T}_k)$ can be identified with a subgroup of $\ker \phi$. Hence, there is a chain of subgroups \[\pi_1(\mathbb{Y}^{(M)}, \mathcal{T}_M) \leq \pi_1(\mathbb{Y}^{(M+1)}, \mathcal{T}_{M+1}) \leq \pi_1(\mathbb{Y}^{(M+2)}, \mathcal{T}_{M+2}) \leq \ldots \] such that \[\ker \phi \cong \varinjlim_{k \geq M} \pi_1(\mathbb{Y}^{(k)}, \mathcal{T}_k).\] 
    
Since the homology functor commutes with direct limits, for each $n \geq 0$ we obtain \begin{equation}\label{eq:homology_direct_limit}H_n(\ker \phi ; \calD_G) \cong \varinjlim_{k\geq M} H_n(\pi_1(\mathbb{Y}^{(k)}, \mathcal{T}_k); \calD_G),\end{equation} where $\calD_G$ is the Linnell ring of $G$. Note that since $G$ is torsion free and satisfies the Atiyah conjecture, the Linnell ring $\calD_G$ is a division ring. 
    
    We claim that for every $k \geq M$, 
\[ \rank_{\calD_G}\, H_n(\pi_1(\mathbb{Y}^{(k)}, \mathcal{T}_k); \calD_G) = \sum_{e\in E^{+}(X)} |\phi(t_e)| \cdot \rank_{\calD_G}H_n(\mathbb{X}_e; \calD_G). \] Assuming that the claim holds, by \eqref{eq:homology_direct_limit} it follows that 
\begin{equation} \begin{split}\label{eq:inequality}\rank_{\calD_G}H_n(\ker \phi; \calD_G) &= \mathrm{rank}_{\calD_G}\varinjlim_{k\geq M} H_n(\pi_1(\mathbb{Y}^{(k)}, \mathcal{T}_k); \calD_G) \\
&\leq \mathrm{lim}_{k \geq M} \rank_{\calD_G}H_n(\pi_1(\mathbb{Y}^{(k)}, \mathcal{T}_k); \calD_G) \\
&= \sum_{e\in E^{+}(X)} |\phi(t_e)| \cdot \rank_{\calD_G}H_n(\mathbb{X}_e; \calD_G). \\ \end{split} \end{equation}

Note that for each edge group $\mathbb{X}_e$, the embedding $\mathbb{X}_e \hookrightarrow G$ extends to an embedding of the Linnell division rings $\calD_{\mathbb{X}_e} \hookrightarrow \calD_G$. Hence, the Linnell ring $\calD_G$ is a free $\calD_{\mathbb{X}_e}$-module and thus
\[  \mathrm{rank}_{\calD_G} H_n(\mathbb{X}_e; \calD_G) = \mathrm{rank}_{\calD_{\mathbb{X}_e}} H_n(\mathbb{X}_e ; \calD_{\mathbb{X}_e}) = b_n^{(2)}(\mathbb{X}_e),
\]
where the second equality follows by \cref{lueck}.
Combining this with \eqref{eq:inequality}, we obtain the inequality
\[ b_n^{(2)}(\ker \phi) \leq \sum_{e\in E^{+}(X)} |\phi(t_e)| \cdot b_n^{(2)}(\mathbb{X}_e).\] This completes the proof modulo the claim.

\medskip

For the claim, note that the graph $Y^{(k)}$ consists of $2k+1$ vertices. By induction, for every edge $e \in E^{+}(X)$, $Y^{(k)}$ contains $k_e$ lifts of the edge $e$, where \[k_e = \max\{0, 2k + 1 - |\phi(t_e)|\}.\] 

For each $e \in E^{+}(X)$, since $b_{n-1}^{(2)}(\mathbb{X}_e) = 0$ by assumption, arguing as above we have that $H_{n-1}(\mathbb{X}_e ; \calD_G)$ is trivial. Moreover, since $b_{n+1}^{(2)}(H) = 0$ for every finitely generated subgroup $H$ of $G$, we have that $H_{n+1}(\pi_1(\mathbb{Y}^{(k)}, \mathcal{T}_k); \calD_G )$ is also trivial. Thus, the Mayer--Vietoris sequence for the graph of groups $\mathbb{Y}^{(k)}$ gives rise to a short exact sequence,
\begin{equation}\label{eq:ses} 0 \to \oplus_{e \in E^{+}(X)} H_n(\mathbb{Y}_e ; \calD_G)^{k_e} \to \oplus_{-k \leq i \leq k}H_n(\mathbb{Y}_{v_i} ; \calD_G) \to H_n(\pi_1(\mathbb{Y}^{(k)}, \mathcal{T}_k); \calD_G ) \to 0.\end{equation}

Note that $\mathbb{Y}_{v_i} \cong \mathbb{X}_v$ for every $i \in \Z$, and $\mathbb{Y}_e \cong \mathbb{X}_e$ for every $e \in E^{+}(X)$. Hence, the short exact sequence \eqref{eq:ses} gives that 
\begin{multline}\label{eq:sum}(2k+1)\cdot \rankD H_n(\mathbb{X}_v ; \calD_G) = \rankD H_n(\pi_1(\mathbb{Y}^{(k)}, \mathcal{T}_k);  \calD_G) \\ + \sum_{e \in E^{+}(X)} k_e \cdot \rankD H_n(\mathbb{X}_e; \calD_G).\end{multline}
By assumption, there is an isomorphism 
\[ \oplus_{e \in E^{+}(X)} H_n(\mathbb{X}_e ; \calD_G) \cong H_n(\mathbb{X}_v; \calD_G).\]
Hence, 
\begin{equation}\label{eq:sum1} \rankD H_n(\mathbb{X}_v ; \calD_G) = \sum_{e\in E^{+}(X)} \rankD H_n(\mathbb{X}_e; \calD_G).\end{equation}
Finally, assuming that $k \geq \mathrm{max}_{e\in E^{+}(X)}\{|\phi(t_e)|\}$, we have that $k_e = 2k + 1 - |\phi(t_e)|$ for all $e \in E^{+}(X)$. Thus, by rearranging \eqref{eq:sum} and using \eqref{eq:sum1}, we arrive at the equality \[\rankD H_n(\pi_1(\mathbb{Y}^{(k)}), \calD_G) =  \sum_{e \in E^{+}(X)} |\phi(t_e)| \cdot \rankD H_n(\mathbb{X}_e; \calD_G). \qedhere\]
\end{proof}

\subsection{Thurston norm for groups}

\begin{definition}[Splitting complexity]\label{def:splitting_complexity}
    Let $G$ be a group and let $\phi \colon G \to \Z$ be a non-trivial homomorphism. The \emph{splitting complexity of $G$ along $\phi$} is defined to be 
    \[ c(G; \phi) = \inf_{\mathbb{X}} c(\mathbb{X}; \phi),\]
    where the infimum is taken over all single-vertex splittings $G \cong \pi_1(\mathbb{X})$ of finite type which are dual to $\phi$. If no such splitting exists then we set $c(G; \phi) = \infty$.
\end{definition}

\begin{lemma}\label{lemma:multiplicative}
    Let $G$ be a group and let $\phi, \psi \colon G \to \Z$ be homomorphisms. Suppose that $\psi$ is non-trivial and $\phi = k \cdot \psi$ for some $k 
    \in \Z \setminus \{0\}$. Then, 
    \[c(G; \phi) = |k| \cdot c(G, \psi).\]
    \end{lemma}

\begin{proof}

Since $\ker \phi = \ker \psi$, any graph-of-groups splitting which is dual to $\phi$ also is dual to $\psi$, and vice versa. Moreover, if $G \cong \pi_1(\mathbb{X})$ is a single-vertex splitting of finite type which is dual to $\phi$, we have that
    \[\begin{split} c(\mathbb{X} ; \phi) &= - \sum_{e \in E^{+}(X)} |\phi(t_e)|\chi(\mathbb{X}_e) \\
    &=  - \sum_{e \in E^{+}(X)} |k| \cdot|\psi(t_e)|\chi(\mathbb{X}_e) \\
    &= |k| \cdot c(\mathbb{X}; \psi)\\
   \end{split}.\]
   This completes the proof.\end{proof}

   \medskip

\begin{theorem}\label{thm:L2_complexity}
    Let $G \not\cong \Z$ be a one-ended coherent \raag and let $\phi \colon G \to \Z$ be an epimorphism. Then the $L^2$-Euler characteristic of the kernel of $\phi$ satisfies
\[-\chi^{(2)}(\ker \phi) = b_1^{(2)}(\ker \phi) = c(G; \phi).\]
\end{theorem}

\begin{proof}
     Let $G = A_L$ for some flag complex $L$. Note that $G$ is torsion-free and satisfies the Atiyah conjecture by \cite{LinnellOkunSchick2012}. Since $G$ is coherent, the flag complex $L$ is chordal by \cref{cor_coherence_characterisation}. Then by \cref{thm_realising_splittings}, there exists a single-vertex splitting of finite type which is dual to $\phi$, $G \cong \pi_1(\mathbb{X})$, such that
    \[ \begin{split}c(\mathbb{X}; \phi) &= - \sum_{v \in V(L)} |\phi(v)|\chi(A_{\mathrm{Lk}(v)}) \\
    &= -\chi^{(2)}(\ker \phi), \end{split}\]
    where the second equality follows by \cref{cor:l2_Euler_norm}.
    This proves the lower bound \[- \chi^{(2)}(\ker \phi) \geq \inf_{\mathbb{X}} c(\mathbb{X}; \phi).\]

    We now prove the upper bound. First note that the $L^2$-Betti numbers of $\ker \phi$ vanish in every degree not equal to one by \cref{prop:l2_betti_numbers_general_kernels}. Hence, the $L^2$-Euler characteristic of the kernel is \[\chi^{(2)}(\ker \phi) = -b_1^{(2)}(\ker \phi).\] Moreover, since the complex $L$ is contractible, it follows that $b_p^{(2)}(G) = 0$ for every $p \geq 0$ by \cref{thm_l2_RAAGs}. Hence $\chi(G) = 0$.

    Let $G \cong \pi_1(\mathbb{X})$ be a single vertex splitting of finite type which is dual to $\phi$. Since $\chi(G) = 0$, we have that 
    \[\chi(\mathbb{X}_v) = \sum_{e \in E^{+}(X)} \chi(\mathbb{X}_e).\] By \cref{prop_subgroup_Betti_numbers}, since the edge and vertex groups of $\mathbb{X}$ are finitely generated, it follows that $\chi(\mathbb{X}_v) = -b_1^{(2)}(\mathbb{X}_v)$ and $\chi(\mathbb{X}_e) = - b_1^{(2)}(\mathbb{X}_e)$ for every edge $e \in E(X)$. Hence, 
    \[ b_1^{(2)}(\mathbb{X}_v) = \sum_{e \in E^{+}(X)} b_1^{(2)}(\mathbb{X}_e).\] Moreover, by \cref{prop_subgroup_Betti_numbers} we have that $b_{0}^{(2)}(\mathbb{X}_e) = 0$ for every edge $e \in E(X)$, and $b_2^{(2)}(H) = 0$ for every finitely generated subgroup $H \leq G$. Thus, the splitting $G \cong \pi_1(\mathbb{X})$ satisfies all the hypotheses of \cref{prop_L2Betti} for $n = 1$. Hence, we have that 
    \[ b_1^{(2)}(\ker \phi) \leq  \sum_{e \in E^{+}(X)} |\phi(t_e)| \cdot \lb_1(\mathbb{X}_e).\] Thus, 
    \[\begin{split} -\chi^{(2)}(\ker \phi) &\leq - \sum_{e \in E^{+}(X)} |\phi(t_e)| \cdot \chi^{(2)}(\mathbb{X}_e) \\
    &= -\sum_{e \in E^{+}(X)} |\phi(t_e)| \cdot \chi(\mathbb{X}_e) ,\end{split}\]
    where the equality in the second line follows from the fact that each $\mathbb{X}_e$ is of type $\F$ and thus $\chi(\mathbb{X}_e) = \chi^{(2)}(\mathbb{X}_e)$. This proves that
    \[ -\chi^{(2)}(\ker \phi) \leq c(\mathbb{X}; \phi).\]
    Hence, 
    \[ - \chi^{(2)}(\ker \phi) \leq \inf_{\mathbb{X}} c(\mathbb{X}; \phi),\]
    where the infimum is taken over all single-vertex splittings of $G$ of finite type which are dual to $\phi$. 
\end{proof}

\medskip

\begin{remark}
    The one-ended assumption in \cref{thm:L2_complexity} is necessary as illustrated by the following example. Let $\phi \colon F_2 \to \Z$ be an epimorphism. Then $\ker \phi$ is isomorphic to the free group of infinite rank and thus $\chi^{(2)}(\ker \phi) = -b_1^{(2)}(\ker \phi)$ is infinite. On the other hand, by \cref{thm_realising_splittings}, there exists a finite splitting $F_2 \cong \pi_1(\mathbb{X})$ which is dual to $\phi$ and such that $c(\mathbb{X}; \phi) \in \Z$.
\end{remark}

\medskip

\begin{theorem}\label{thm_main}
Let $G$ be a one-ended coherent \raagstop There is a canonical continuous function 
    \[ \| \cdot \|_T \colon H^1(G; \RR) \to \RR \]
    which satisfies the following properties.
    \begin{enumerate}
        \item The function $\| \cdot \|_T$ is convex and linear on rays through the origin.
        \item For any non-trivial integral character $\phi \in H^1(G ; \Z)$, we have that $\| \phi \|_T = c(G; \phi)$ where $c(G; \phi)$ is the \emph{splitting complexity} of $\phi$.
        \item In particular, for any fibered epimorphism $\phi \colon G \to \Z$, we have that $\| \phi \|_T = \chi(\ker \phi)$, where $\chi(\ker \phi)$ is the Euler characteristic of $\ker \phi$.
    \end{enumerate}
\end{theorem}

\begin{proof}

    By \cref{lemma:L2_polytope_RAAG}, the $L^2$-polytope $\mathcal{P}(G)$ exists and is a single polytope. Hence, by \cref{lemma:polytope_semi-norm}, the map 
    \[ \begin{split} H^1(G; \RR) &\to \RR  \\
    \phi &\mapsto \mathrm{th}_{\phi}(\mathcal{P}(G))\end{split} \]
    determines a semi-norm on $H^1(G; \RR)$. We define $\| \phi \|_T = \mathrm{th}_{\phi}(\mathcal{P}(G))$ for all $\phi \in H^1(G; \RR)$.
    
    By \cref{prop:L2Euler} and \cref{thm:L2_complexity}, if $\phi \colon G \to \Z$ is an epimorphism then 
    \begin{equation}\label{eq:L2_Euler_equality}\mathrm{th}_{\phi}(\mathcal{P}(G)) = - \chi^{(2)}(\ker \phi) = c(G; \phi).\end{equation} Moreover, by \cref{lemma:multiplicative}, it follows that for every non-zero integer $k \in \Z$, 
    \[c(G, k \cdot \phi) = |k| \cdot c(G; \phi).\]
    Hence, for any non-trivial homomorphism $\phi \colon G \to \Z$,
    \[ \| \phi \|_T = c(G; \phi).\]

    Then, it follows by \cref{cor_coherence_characterisation} that if the kernel $\ker \phi$ of an epimorphism $\phi \colon G \to \Z$ is finitely generated, then it is of type $\F$. Hence, 
    \[\chi^{(2)}(\ker \phi) = \chi(\ker \phi), \]
    and the final part of the Theorem follows by \eqref{eq:L2_Euler_equality}.

    Finally, we note that there is a unique function $\| \cdot \|_T$ with the given properties, since every such function is determined by $c(G; \phi)$ on a dense subset of $H^{1}(G;\RR)$. Moreover, for any automorphism $\Theta \colon G \to G$, we have that 
    \[ \chi^{(2)}(\ker (\phi \circ \Theta)) = \chi^{(2)}(\Theta^{-1}(\ker \phi)) = \chi^{(2)}(\ker \phi).\] Hence it must be the case that $\| \cdot \|_T$ is $\Theta$-invariant.
    
    \end{proof}

\bibliographystyle{alpha}
\bibliography{refs}

\end{document}